\documentclass{article}
\usepackage[dvipsnames]{xcolor}
\usepackage{gastex}
\usepackage{amsmath}
\usepackage{amssymb}
\usepackage{upref}
\usepackage{multirow}
\usepackage{multicol}
\usepackage{url}

\usepackage{theorem}
\newtheorem{theorem}{Theorem}
\newtheorem{lemma}[theorem]{Lemma}
\newtheorem{proposition}[theorem]{Proposition}
\newtheorem{corollary}[theorem]{Corollary}
{\theorembodyfont{\rmfamily}%
  \newtheorem{example}[theorem]{Example}
   }
\newenvironment{proof}{\noindent\textit{Proof.}}
{\QED\vskip\theorempostskipamount} 
\newenvironment{proofof}[1]{\noindent\textit{Proof
    \protect{#1}.}}
                       {\QED\vskip\theorempostskipamount}
\def\petitcarre{\vrule height4pt width 4pt depth0pt}
\def\QED{\relax\ifmmode\eqno{\hbox{\petitcarre}}\else{%
  \unskip\nobreak\hfil\penalty50\hskip2em\hbox{}\nobreak\hfil
  \petitcarre
  \parfillskip=0pt \finalhyphendemerits=0\par\smallskip}
  \fi}

\newcommand\A{\mathcal{A}}
\newcommand\B{\mathcal{B}}

\newcommand\RR{\mathcal{R}}

\newcommand{\Z}{\mathbb{Z}}

\let\edge\xrightarrow
\def\un(#1){\underline{#1}\,}
\DeclareMathOperator{\Card}{Card}

\definecolor{ivoire}{rgb}{0.99,0.99,0.8}
\usepackage{calc}
\newcounter{hours}\newcounter{minutes}
\newcommand\computetime{\setcounter{hours}{\time/60}%
  \setcounter{minutes}{\time-\value{hours}*60}%
  \thehours\,h\,\theminutes}
\newcommand\dateandtime{\today\quad\computetime}
\numberwithin{theorem}{section}
\numberwithin{equation}{section}
\numberwithin{figure}{section}
\numberwithin{table}{section}
\title{Acyclic, connected and tree sets}
\author{Val\'erie Berth\'e$^1$, Clelia De Felice$^2$, 
Francesco Dolce$^3$, Julien Leroy$^4$,\\
 Dominique Perrin$^3$,
Christophe  Reutenauer$^5$,
Giuseppina Rindone$^3$\\\\
$^1$CNRS, Universit\'e Paris 7,
$^2$Universit\`a degli Studi di Salerno,\\
$^3$Universit\'e Paris Est, LIGM,
$^4$Universit\'e du Luxembourg,\\
 $^5$Universit\'e du Qu\'ebec \`a Montr\'eal}
\date{\dateandtime}

\makeatletter
\def\@listI{%
  \leftmargin\leftmargini
  \setlength{\parsep}{0pt plus 1pt minus 1pt}
  \setlength{\topsep}{2pt plus 1pt minus 1pt}
  \setlength{\itemsep}{0pt}
}
\let\@listi\@listI
\@listi
\def\@listii {%
  \leftmargin\leftmarginii
  \labelwidth\leftmarginii
  \advance\labelwidth-\labelsep
  \setlength{\topsep}{0pt plus 1pt minus 1pt}
}
\def\@listiii{%
  \leftmargin\leftmarginiii
  \labelwidth\leftmarginiii
  \advance\labelwidth-\labelsep
  \setlength{\topsep}{0pt plus 1pt minus 1pt}
  \setlength{\parsep}{0pt} 
  \setlength{\partopsep}{1pt plus 0pt minus 1pt}
}
\makeatother

\begin{document}

\maketitle

\begin{abstract}
Given a set $S$ of words,
one associates to each word $w$ in $S$ an undirected graph, called its extension
graph, and which describes the possible extensions of $w$ in $S$
on the left and on the right.
We investigate the family of sets of words defined
by the property of the extension graph of each word in the set
to be acyclic or connected or a tree. We exhibit for this family
various connexions between word combinatorics, bifix codes, group automata and
free groups. We prove that in a uniformly
recurrent tree set, the sets of first return words are bases
of the free group on the alphabet. Concerning acyclic sets,
we prove as a main result that a set $S$ is
acyclic if and only if any  bifix code included in $S$ is a basis
of the subgroup that it generates.
\end{abstract}
\paragraph{AMS-Classification-Numbers} 05A Enumerative Combinatorics,
37B Topological Dynamics
\tableofcontents
\section{Introduction}
This paper studies properties of  classes of sets which occur
as the set of factors of infinite words
of linear factor complexity. It is part of a series of
papers devoted to this subject initiated
in~\cite{BerstelDeFelicePerrinReutenauerRindone2012}. These classes of sets,
called acyclic, connected or tree sets,
are defined by a limitation to the possible two-sided
extensions of a word of the set. We will see that Sturmian sets
 are tree sets (by Sturmian we mean
the sets of factors of strict episturmian words, also called
Arnoux-Rauzy words). Moreover, the sets obtained by
coding a regular interval exchange set are also tree sets
(see \cite{BertheDeFeliceDolceLeroyPerrinReutenauerRindone2013}).
Any word $w$ in a tree set is neutral in the sense that the number
of pairs $(a,b)$ of letters such that $awb\in S$ is equal to
the number of letters $a$ such that $aw\in S$ plus the number
of letters $b$ such that $wb\in S$ minus $1$. We express this property
saying that it is a neutral set.

We study sets of first return words in a tree set $S$.
Our main result on return words is that
if $S$ is a uniformly recurrent
tree set containing the alphabet $A$, the set of first return words
to any word of $S$ is a basis of the free group on $A$
(Theorem~\ref{corollaryJulien}  referred to as the Return Theorem).
 For this, we use Rauzy graphs, obtained by restricting
  de Bruijn graphs to the set of vertices formed by the words of
given length in a set $S$. We first show
that if $S$ is a recurrent connected set  containing the alphabet $A$, the group described by
any Rauzy graph of $S$
with respect to some vertex is the free group
on $A$ (Theorem~\ref{proposition3}). Next, we prove that
in a uniformly recurrent connected set $S$ containing $A$,
the set of first return words to any word in $S$ generates the free
group 
on $A$ (Theorem~\ref{theoremJulien}). The proof uses the fact
that in a uniformly recurrent neutral set $S$ containing the alphabet $A$,
the number of first return words to any word of $S$ is equal to
$\Card(A)$,
a result obtained in~\cite{BalkovaPelantovaSteiner2008}.

We also study bifix codes in acyclic sets.
Our main result is that a set $S$ is acyclic if and only if
any bifix code contained in $S$ 
is a basis of the subgroup that it generates
(Theorem~\ref{basisTheorem} referred to as the Freeness Theorem). This is related to the
main result of~\cite{BerstelDeFelicePerrinReutenauerRindone2012},
referred to as the Finite Index Basis Theorem,
proving that, in a Sturmian set $S$, a finite bifix code is
$S$-maximal of $S$-degree $d$ if and only if it is a basis
of a subgroup of index $d$. This result is generalized
in~\cite{BertheDeFeliceDolceLeroyPerrinReutenauerRindone2013} to uniformly
recurrent tree sets. The proof uses the results of this paper and,
in particular the Return Theorem (Theorem~\ref{corollaryJulien}).
In the case of an acyclic set, the subgroup generated by a bifix code need
not be of finite index, even if the bifix code is $S$-maximal
(and even if the set $S$ is uniformly recurrent, see
Example~\ref{exampleBasisJulien}).

We also prove a more technical result. We say that a submonoid $M$
of the free monoid is saturated in a set $S$ if the subgroup $H$
of the free group generated by $M$ satisfies $M\cap S=H\cap S$.
We prove that if $S$ is acyclic, the submonoid generated by a bifix
code
contained in $S$ is saturated in $S$ (Theorem~\ref{saturationTheorem}
referred to as the Saturation Theorem). This property plays an
important role in the proof of the Finite Index Basis Theorem.

Our paper is organized as follows. 

In
Section~\ref{sectionPreliminaries}
we present the definitions and basic properties used in the paper.
We introduce strong, weak and
neutral sets.  We 
 prove a result
on the cardinality of sets of first return words 
(Theorem~\ref{theoremCardReturn}) which
is a generalization of a result from~\cite{BalkovaPelantovaSteiner2008}.

In Section~\ref{sectionAcyclic}, we define the extension graph
of a word with respect to a set $S$. This notion appears
already in~\cite{PelantovaStarosta2014} with a purpose  similar to
ours. We define acyclic, connected
and tree sets
by the corresponding property of the extension graph of each
word in the set to be acyclic, connected or a tree. We also
introduce more general extension
graphs where left (resp. right) extensions are relative to a finite
suffix (resp. prefix) code. We prove that in acyclic sets,
these more general extension graphs  are
also acyclic (Proposition~\ref{PropStrongTreeCondition}).

In Section~\ref{sectionReturnTreeSets}, we study sets of first return words in tree sets. We first show
that if $S$ is a recurrent connected set containing the alphabet $A$,
 the group described by
any Rauzy graph of $S$,
with respect to some vertex is the free group
on $A$ (Theorem~\ref{proposition3}). Next, we prove that
in a uniformly recurrent connected set $S$
containing $A$, the set of first return words
to any word of $S$ generates the free group on $A$
(Theorem~\ref{theoremJulien}). We use Theorem~\ref{theoremCardReturn}
to prove that if $S$ is additionally acyclic, then every set
of first return words is a basis of the free group on $A$
(Theorem~\ref{corollaryJulien}).

In Section~\ref{sectionMainResult} we state and prove our main results
(Theorem~\ref{basisTheorem} and Theorem~\ref{saturationTheorem}).
The proof uses the notion of incidence graph of a bifix code
(already introduced
in~\cite{BerstelDeFelicePerrinReutenauerRindone2012}).

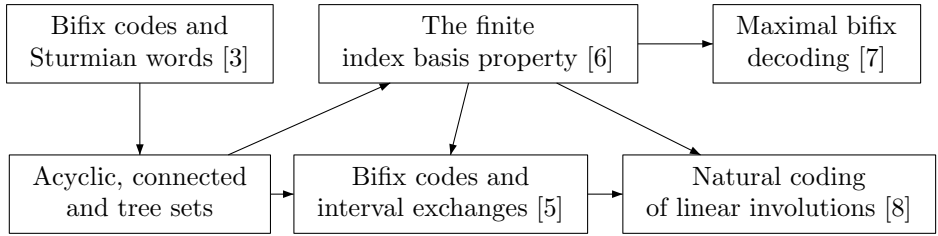
\begin{figure}[hbt]
\gasset{Nmr=0,Nadjust=wh}
\begin{picture}(120,30)(0,-5)
\node(BS)(15,20)
{\begin{tabular}{c}Bifix codes and \\ Sturmian words~\cite{BerstelDeFelicePerrinReutenauerRindone2012}\end{tabular}}
\node(ACT)(15,0){\begin{tabular}{c}Acyclic, connected\\ and tree sets\end{tabular}}
\node(BIE)(55,0){\begin{tabular}{c}Bifix codes and\\
 interval exchanges~\cite{BertheDeFeliceDolceLeroyPerrinReutenauerRindone2014}\end{tabular}}
\node(FIB)(60,20){\begin{tabular}{c}The finite \\index basis property~\cite{BertheDeFeliceDolceLeroyPerrinReutenauerRindone2013}\end{tabular}}
\node(MBD)(105,20){\begin{tabular}{c}Maximal bifix\\
    decoding~\cite{BertheDeFeliceDolceLeroyPerrinReutenauerRindone2013b}\end{tabular}}
\node(NCLI)(100,0){\begin{tabular}{c}Natural coding\\ of linear involutions~\cite{BertheDolceLeroyPerrinReutenauerRindone2013c}\end{tabular}}
\drawedge(BS,ACT){}
\drawedge(ACT,FIB){}\drawedge(ACT,BIE){}
\drawedge(FIB,MBD){}
\drawedge(BIE,NCLI){}
\drawedge(FIB,BIE){}
\drawedge(FIB,NCLI){}
\end{picture}
\end{figure}
Some  results used in this paper are proved in our first paper
\cite{BerstelDeFelicePerrinReutenauerRindone2012}. In turn, the
results of this paper are used in other papers in preparation
on similar objects. We include for clarity the logical
dependency between these papers.

\paragraph{Ackowledgement} This work was supported by grants from
Region Ile-de-France, the ANR projects Dyna3S and Eqinocs, the FARB Project
``Aspetti algebrici e computazionali nella teoria dei codici,
degli automi e dei linguaggi formali'' (University of Salerno, 2013)
and the MIUR PRIN 2010-2011 grant
``Automata and Formal Languages: Mathematical and Applicative Aspects''.

The authors thank the referee for his suggestions which helped
to improve the presentation of our paper.

\section{Preliminaries}\label{sectionPreliminaries}
In this section, we first recall some definitions concerning words,
codes and automata (see~\cite{BerstelPerrinReutenauer2009}
 for a more complete presentation).
We give the definition of recurrent and uniformly recurrent
sets of words. We also give the definitions and basic
properties of bifix codes (see~\cite{BerstelDeFelicePerrinReutenauerRindone2012} for a more detailed
presentation). We define basic notions concerning automata.
We present the class of reversible automata and its connection
with the Stallings automaton of a subgroup of a free group.
We finally introduce strong, weak and neutral sets
and state some results concerning the factor complexity of these sets.
 We also introduce return words and we
 recall a result from~\cite{BalkovaPelantovaSteiner2008}
on the cardinality of sets of first return words 
(Theorem~\ref{theoremCardReturn}) we is used later.

\subsection{Recurrent sets}
Let $A$ be a finite nonempty alphabet. All words considered below,
unless
stated explicitly, are supposed to be on the alphabet $A$.
We denote by $A^*$ the set of all words on $A$.
We denote by $1$ or by $\varepsilon$ the empty word.
We denote by $|x|$ the length of a word $x$.
A set of words is said to be \emph{factorial} if it contains the
factors of its elements.

For a set $X$ of words and a word $u$, we denote
\begin{displaymath}
u^{-1}X=\{v\in A^*\mid uv\in X\}.
\end{displaymath}
the right \emph{residual} of $X$ with respect to $u$.

Let $S$ be a set of words on the alphabet $A$.
 For $w\in S$,
we denote
\begin{displaymath}
L(w)=\{a\in A\mid aw\in S\},\quad
R(w)=\{a\in A\mid wa\in S\}
\end{displaymath}
\begin{displaymath}
E(w)=\{(a,b)\in A\times A\mid awb\in S\}
\end{displaymath}
and further $\ell(w)=\Card(L(w))$, $r(w)=\Card(R(w))$, $e(w)=\Card(E(w))$.

A word $w$ is \emph{right-extendable} if $r(w)>0$,
\emph{left-extendable} if $\ell(w)>0$ and \emph{biextendable} if
$e(w)>0$. A 
factorial set
$S$ is called \emph{right-extendable}
(resp. \emph{left-extendable}, resp. \emph{biextendable}) if every word in $S$ is
right-extendable (resp. left-extendable, resp. biextendable).

A word $w$ is called
\emph{right-special}
if $r(w)\ge 2$. It is called \emph{left-special} if $\ell(w)\ge 2$.
It is called \emph{bispecial} if it is both right and left-special.

A set of words $S\ne\{1\}$ is \emph{recurrent} if it is factorial and if for every
$u,w\in S$ there is a $v\in S$ such that $uvw\in S$. A recurrent set
 is biextendable.

A set of words $S$ is said to be \emph{uniformly recurrent} if it is
right-extendable and if, for any word $u\in S$, there exists an integer $n\ge
1$
such that $u$ is a factor of every word of $S$ of length $n$.
A uniformly recurrent set is recurrent.

A \emph{morphism} $f:A^*\rightarrow B^*$ is a monoid morphism from
$A^*$ into $B^*$. If $a\in A$ is such that the word $f(a)$ begins with
$a$ and if $|f^n(a)|$ tends to infinity with $n$, there is a unique
infinite word denoted $f^\omega(a)$ which has all words $f^n(a)$
as prefixes. It is called a \emph{fixpoint} of the morphism $f$.

A morphism $f:A^*\rightarrow A^*$ is called \emph{primitive} if there
is an integer $k$ such that for all $a,b\in A$, the letter $b$
appears in $f^k(a)$. If $f$ is a primitive morphism, the set
of factors of any fixpoint
of $f$ is uniformly recurrent (see~\cite{PytheasFogg2002}
Proposition 1.2.3 for example).

An infinite word is \emph{episturmian} if the set of its factors 
is closed under reversal and contains for each $n$ at most one word
of length $n$ which is right-special (see~\cite{BerstelDeFelicePerrinReutenauerRindone2012} for more references). It is a \emph{strict episturmian} word
if it has exactly one right-special word of each length and
moreover each right-special factor $u$ is such that $r(u)=\Card(A)$.


A \emph{Sturmian set} is a set of words which 
is the set of factors of a strict episturmian word.
Any Sturmian set is uniformly recurrent (see~\cite{BerstelDeFelicePerrinReutenauerRindone2012}).
\begin{example}\label{exampleFibonacci}
Let $A=\{a,b\}$.
The \emph{Fibonacci morphism} is the morphism
 $f:A^*\rightarrow A^*$ defined by $f(a)=ab$ and $f(b)=a$.
The \emph{Fibonacci word}
\begin{displaymath}
x=abaababaabaababaababa\ldots
\end{displaymath}
is the fixpoint $x=f^\omega(a)$ of the
Fibonacci morphism.
It is a Sturmian word (see~\cite{Lothaire2002}). The set $F(x)$ of factors of $x$ is the
\emph{Fibonacci set}.
\end{example}
\begin{example}\label{exampleTribonacci}
Let $A=\{a,b,c\}$.
The Tribonacci word 
\begin{displaymath}
x=abacabaabacababacabaabacaba\cdots
\end{displaymath}
is the fixpoint $x=f^\omega(a)$ of the morphism
$f:A^*\rightarrow A^*$ defined by $f(a)=ab$, $f(b)=ac$, $f(c)=a$.
It is a strict episturmian word (see~\cite{JustinVuillon2000}).
The set $F(x)$ of factors of $x$ is the \emph{Tribonacci set}.
\end{example}

We fix our notation concerning free groups
(see~\cite{LyndonSchupp2001} for
example).

We denote by $F_A$ the free group on the alphabet $A$. It is
identified with
the set of all words on the alphabet $A\cup A^{-1}$ which are
\emph{reduced}, in the sense that they do not have any factor $aa^{-1}$
or $a^{-1}a$ for $a\in A$. Note that the exponent $-1$
used in this context should not be confused with the one used to define
the residual of a set of words. We extend the bijection
$a\mapsto a^{-1}$ to an involution on $A\cup A^{-1}$ by
defining $(a^{-1})^{-1}=a$.

For any word $w$ on $A\cup A^{-1}$ there is a unique reduced
word equivalent to $w$ modulo the relations $aa^{-1}\equiv a^{-1}a\equiv 1$
for $a\in A$. If $u$ is the reduced word
equivalent to $w$, we say that $w$ \emph{reduces} to $u$
and we denote $w\equiv u$. We also denote $u=\rho(w)$.
The product of two elements $u,v\in F_A$
is the reduced word $w$ equivalent to $uv$, namely $\rho(uv)$.

For a set $X$ of reduced words, we denote $X^{-1}=\{x^{-1}\mid x\in X\}$.

\subsection{Bifix codes}

A \emph{prefix code} is a set of nonempty words which does not contain any
proper prefix of its elements. A suffix code is defined symmetrically.
A  \emph{bifix code} is a set which is both a prefix code and a suffix
code.

We denote by $X^*$ the submonoid generated by a set $X$ of words. The submonoid $M$
generated
by a prefix code satisfies the following property: if $u,uv\in M$,
then $v\in M$. Such a submonoid is said to be \emph{right unitary}.
The definition of a left unitary submonoid is symmetric
and the submonoid generated by a suffix code is left unitary.
Conversely, any right unitary (resp. left unitary) submonoid of $A^*$
is 
generated by a unique prefix code (resp. suffix code)
(see~\cite{BerstelPerrinReutenauer2009}).

A \emph{coding morphism} for a prefix code $X\subset A^+$ is a morphism
$f:B^*\rightarrow A^*$ which maps bijectively $B$ onto $X$ (note that in this paper we use $\subset$ to denote
the inclusion allowing equality).

Let $S$ be a set of words. A prefix code $X\subset S$ is $S$-maximal 
if it is not properly contained in any prefix code 
$Y\subset S$.

A set $X\subset S$ is \emph{right $S$-complete} if any word of $S$
is a prefix of a word in $X^*$. 

For a factorial set $S$, a prefix code is $S$-maximal if
and only if it is right $S$-complete (Proposition 3.3.2
in~\cite{BerstelDeFelicePerrinReutenauerRindone2012}).

Similarly a bifix code $X\subset S$ is $S$-maximal if it is
not properly contained in a bifix code $Y\subset S$.
For a recurrent set $S$, a finite bifix code is $S$-maximal as a bifix code if
and only if it is an $S$-maximal prefix code 
(see~\cite{BerstelDeFelicePerrinReutenauerRindone2012}, Theorem
4.2.2). For a uniformly recurrent set $S$, any finite bifix code 
$X\subset S$ is contained in a finite $S$-maximal bifix code
(Theorem 4.4.3 in~\cite{BerstelDeFelicePerrinReutenauerRindone2012}).

A \emph{parse} of a word $w$ with respect to a bifix code $X$ is
a triple $(v,x,u)$ such that $w=vxu$ where $v$ has no suffix in $X$,
$u$ has no prefix in $X$ and $x\in X^*$.
We denote by $d_X(w)$ the number of parses of $w$.
By definition, the $S$-\emph{degree} of $X$, denoted $d_X(S)$, is the maximal number
of parses of a word in $S$.  It can be finite or infinite.

Let $X$ be a bifix code.
The number of parses of a word $w$ is also equal to the number
of suffixes of $w$ which have no prefix in $X$ and to the
 number of prefixes of $w$ which have no suffix in $X$
(see Proposition 6.1.6 in~\cite{BerstelPerrinReutenauer2009}).

The set of \emph{internal factors} of a set of words $X$,
denoted $I(X)$
is the set of words $w$ such that 
there exist nonempty words $u,v$ with $uwv\in X$.

Let $S$ be a recurrent set and let
 $X$ be a finite  bifix code.
By Theorem 4.2.8 in~\cite{BerstelDeFelicePerrinReutenauerRindone2012},
$X$ is $S$-maximal if and only if its $S$-degree $d$ is finite. Moreover, in
this case,
a word $w\in S$ is such that
$d_X(w)< d$ if and only if it is an internal factor of $X$,
that is
\begin{displaymath}
I(X)=\{w\in S\mid d_X(w)<d\}.
\end{displaymath}
In particular, any word of $X$ of maximal length has $d$ parses.

\begin{example}\label{exampleUniform}
Let $S$ be a recurrent set. For any integer $n\ge 1$, the set
$S\cap A^n$ is an $S$-maximal bifix code of $S$-degree $n$.
\end{example}

\subsection{Automata and groups}\label{sectionAutomata}
We denote $\A=(Q,i,T)$ a deterministic automaton with a  set
$Q$  of
states,
$i\in Q$ as initial state and $T\subset Q$ as set of terminal states.
For $p\in Q$ and $w\in A^*$, we denote $p\cdot w=q$ if
 there is a path labeled $w$ from $p$ to the state $q$ and $p\cdot
 w=\emptyset$
otherwise. The automaton is \emph{finite} when $Q$ is finite.

The set \emph{recognized}\index{recognized by an automaton} by the
automaton is the set of words $w\in A^*$ such that $i\cdot w\in T$. 

All automata considered in this paper are deterministic and we simply
call them  `automata' to mean `deterministic automata'.

The automaton $\A$ is \emph{trim}\index{automaton!trim}\index{trim
  automaton} if for any $q\in Q$, there is a path from $i$ to $q$ and
a path from $q$ to some $t\in T$.

An automaton is called
\emph{simple}\index{automaton!simple}\index{simple automaton} if it is
trim and if it has a unique terminal state which coincides with the
initial state.
The set recognized by a simple automaton is a right unitary
submonoid. Thus it is generated by a
prefix code.

An automaton $\A=(Q,i,T)$ is \emph{complete}%
\index{automaton!complete}\index{complete automaton}
if for any state $p\in Q$
and
any letter $a\in A$, one has $p\cdot a\ne\emptyset$.

For a nonempty set $L\subset A^*$, we denote by $\A(L)$ the \emph{minimal
  automaton}\index{automaton!minimal}\index{minimal automaton}
 of $L$. The states of $\A(L)$ are the nonempty residuals
$u^{-1}L$ for $u\in A^*$.  
For $u\in A^*$ and $a\in A$, one defines $(u^{-1}L)\cdot a=(ua)^{-1}L$.
The initial state
is the set $L$ itself and the terminal states are the sets $u^{-1}L$ for
$u\in L$.


Let $X$ be a prefix code and let $P$ be the set of proper prefixes of
$X$. The \emph{literal automaton}\index{literal
  automaton}\index{automaton!literal} of $X^*$ is the simple automaton
$\A=(P,1,1)$ with transitions defined for $p\in P$ and $a\in A$ by
\begin{displaymath}
  p\cdot a=
  \begin{cases}
    pa&\text{if $pa\in P$}\,,\\
    1&\text{if $pa\in X$}\,,\\
    \emptyset&\text{otherwise}.
  \end{cases}
\end{displaymath}
One verifies that this automaton recognizes $X^*$.
Thus for any prefix code $X\subset A^*$, there is a simple automaton
$\A=(Q,1,1)$ which recognizes $X^*$. Moreover, the minimal
automaton of $X^*$ is simple. Note that
the literal automaton is not minimal in
general (see Example~\ref{exampleLiteral}).
\begin{example}\label{exampleLiteral}
Let $X=\{aa,ab,bba,bbb\}$. The literal and the minimal automata
of $X^*$ are represented in Figure~\ref{figLiteralMinimal}
(the initial state is indicated by an incoming arrow and the terminal
states
by an outgoing one).
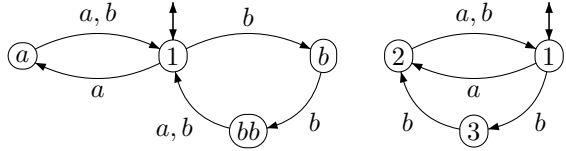
\begin{figure}[hbt]
\centering
\gasset{Nadjust=wh}
\begin{picture}(70,20)
\put(0,0){
\begin{picture}(30,20)
\node(2)(0,10){$a$}
\node[Nmarks=if,iangle=90,fangle=90](1)(20,10){$1$}
\node(3)(40,10){$b$}\node(4)(30,0){$bb$}

\drawedge[curvedepth=3](1,2){$a$}\drawedge[curvedepth=3](2,1){$a,b$}
\drawedge[curvedepth=3](1,3){$b$}
\drawedge[curvedepth=3](3,4){$b$}\drawedge[curvedepth=3](4,1){$a,b$}
\end{picture}
}
\put(50,0){
\begin{picture}(40,20)
\node(2)(0,10){$2$}
\node[Nmarks=if,iangle=90,fangle=90](1)(20,10){$1$}
\node(3)(10,0){$3$}

\drawedge[curvedepth=3](1,2){$a$}\drawedge[curvedepth=3](2,1){$a,b$}
\drawedge[curvedepth=3](1,3){$b$}\drawedge[curvedepth=3](3,2){$b$}

\end{picture}
}
\end{picture}
\caption{The literal and the minimal automata of $X^*$.}\label{figLiteralMinimal}
\end{figure}
\end{example}
A simple automaton $\A=(Q,1,1)$ is said to be
\emph{reversible}\index{automaton!reversible}\index{reversible
  automaton} if for any $a\in A$, the partial map
$\varphi_\A(a):p\mapsto p\cdot a$ is injective. This condition allows
to construct the \emph{reversal} of the automaton\index{reversal of an
  automaton} as follows: whenever $q\cdot a =p$ in $\A$, then $p\cdot
a =q$ in the reversal automaton. The state~$1$ is the initial and the
unique terminal state of this automaton.  Any reversible automaton is
minimal~\cite{Reutenauer1979} (but not conversely). 
The set recognized by a reversible
automaton  is a submonoid generated by a bifix code.

A simple automaton $\A=(Q,1,1)$ is a \emph{group automaton}
if for any $a\in A$ the map $\varphi_\A(a):p\mapsto p\cdot a$
is a permutation of $Q$. Thus in particular, a group automaton
is reversible.
 A finite reversible automaton which is complete is a group automaton.

 The
following result is from~\cite{Reutenauer1979} (see also Exercise
6.1.2 in \cite{BerstelPerrinReutenauer2009}). We denote by
$\langle X\rangle$ the subgroup of the free group $F_A$ generated by $X$.

\begin{proposition}\label{lemmaExercise612}
  Let $X\subset A^+$ be a bifix code.  The following conditions are
  equiva\-lent.
  \begin{enumerate}
  \item[\upshape{(i)}] $X^*=\langle X\rangle\cap A^*$;
  \item[\upshape{(ii)}] the minimal automaton of $X^*$ is reversible.
  \end{enumerate}
\end{proposition}

Let $\A=(Q,i,T)$ be a deterministic automaton. A \emph{generalized
  path} \index{generalized path}\index{path!generalized} is a sequence
$(p_0,a_1,p_1,a_2,\ldots,p_{n-1},a_n,p_n)$ with $a_i\in A\cup A^{-1}$
and $p_i\in Q$, such that for $1\le i\le n$, one has $p_{i-1}\cdot
a_i=p_i$ if $a_i\in A$ and $p_{i}\cdot a_i^{-1}=p_{i-1}$ if $a_i\in
A^{-1}$.  The \emph{label} of the generalized path is  the reduced
word
equivalent to
$a_1a_2\cdots a_n$. It is an element
 of the free group $F_A$. The set \emph{described} by the automaton
is the set of labels of generalized paths from $i$ to a state in $T$.
Since a path  is a particular case of a generalized
path, the set recognized by an automaton $\A$ is a subset
of the set described by $\A$.

The set described by a simple automaton is a
subgroup of $F_A$. It is called the \emph{subgroup
  described}\index{subgroup described} by $\A$.

\begin{example}\label{exGroupRecognized}
  Let $\A=(Q,1,1)$ be the automaton represented in Figure~\ref{figDescribed}. 
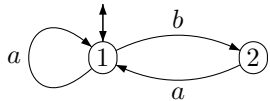
\begin{figure}[hbt]
\gasset{Nadjust=wh}\centering
\begin{picture}(20,10)
\node[Nmarks=if,iangle=90,fangle=90](1)(0,5){$1$}\node(2)(20,5){$2$}
\drawloop[loopangle=180](1){$a$}
\drawedge[curvedepth=3](1,2){$b$}\drawedge[curvedepth=3](2,1){$a$}
\end{picture}
\caption{A simple automaton describing the free group on $\{a,b\}$.}\label{figDescribed}
\end{figure}
The submonoid
  recognized by $\A$ is $\{a,ba\}^*$.  Since $\{a,ba\}$
is a basis of the free group on $A$, the subgroup described by $\A$ is
  the free group on $A$.
\end{example}

The following result is Proposition 6.1.3
 in~\cite{BerstelDeFelicePerrinReutenauerRindone2012}.

\begin{proposition}\label{propGeneratedGroup}
  Let $\A$ be a simple automaton and let $X$ be the prefix code
  generating the submonoid recognized by $\A$. The subgroup  described
  by $\A$ is generated by $X$. If moreover $\A$ is reversible, then
  $X^*=\langle X\rangle\cap A^*$.
\end{proposition}

For any subgroup $H$ of $F_A$, the submonoid $H\cap A^*$ 
 is right and left unitary and thus it is
generated by a bifix code (see~\cite{BerstelPerrinReutenauer2009},
Example 2.2.6).
A subgroup $H$ of the free group on $A$ is \emph{positively
  generated}\index{subgroup!positively generated} if there is a subset of
$A^*$ which generates $H$. In this case, the set $H\cap A^*$
generates the subgroup $H$. Let $X$ be the bifix code which generates
the submonoid $H\cap A^*$. Then $X$ generates the subgroup $H$. This
shows that, for a positively generated subgroup $H$, there is a bifix
code which generates $H$. 

It is well known that a subgroup of finite index of the free group is 
positively generated (see e.g. Proposition 6.1.6
in~\cite{BerstelDeFelicePerrinReutenauerRindone2012}).

The following result is contained in
Proposition 6.1.4 and 6.1.5 
in~\cite{BerstelDeFelicePerrinReutenauerRindone2012}.

\begin{proposition}\label{propStallings}
  For any positively generated subgroup $H$ of the free group on $A$, there is a
  unique reversible  automaton $\A$ such that $H$ is the
  subgroup described by $\A$. The subgroup is of finite index
if and only if this automaton is a finite group automaton.
\end{proposition}

For an illustration, see Example~\ref{exampleBasisJulien} below.

The reversible automaton $\A$ such that $H$ is the subgroup described
by $\A$ is called the \emph{Stallings automaton}\index{Stallings
  automaton} of the subgroup $H$. It can also be defined for a
subgroup which is not positively generated
(see~\cite{BartholdiSilva2011} or \cite{KapovichMyasnikov2002}).

The Stallings automaton of the subgroup $H$ generated by a bifix
code $X\subset A^*$ can be obtained as follows. Start with the
minimal automaton $\A=(Q,1,1)$ of $X^*$. Then, if there are distinct states
$p,q\in Q$ and $a\in A$ such that $p\cdot a=q\cdot a$, merge
$p,q$ (such a merge is called a \emph{Stallings folding}).
Iterating this operation leads to a reversible automaton
which is the Stallings automaton of $H$
(see~\cite{KapovichMyasnikov2002}).

 A subgroup $H$ of the free group has finite index if and only if
its Stallings automaton is a finite group automaton (see
Proposition~\ref{propStallings}). In this case, the index of $H$
is the number of states of the Stallings automaton.
\begin{example}
Let $X=\{aa,ab,ba\}$.
The minimal automaton of $X^*$ is represented in Figure~\ref{figReversible}
on the left. It is not reversible because $2\cdot a=3\cdot a$.
 Merging the states $2$ and $3$, we obtain the reversible
automaton of Figure~\ref{figReversible} on the right.
It is actually a group automaton, which
 is the Stallings automaton of the subgroup $H=\langle X\rangle$.
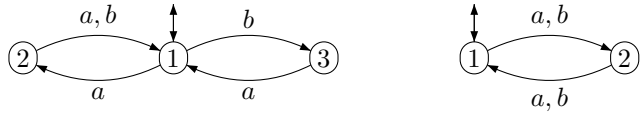
\begin{figure}[hbt]
\gasset{Nadjust=wh}\centering
\begin{picture}(80,12)
\put(0,0){
\begin{picture}(40,10)
\node(2)(0,5){$2$}
\node[Nmarks=if,iangle=90,fangle=90](1)(20,5){$1$}\node(3)(40,5){$3$}
\drawedge[curvedepth=3](1,2){$a$}\drawedge[curvedepth=3](2,1){$a,b$}
\drawedge[curvedepth=3](1,3){$b$}\drawedge[curvedepth=3](3,1){$a$}
\end{picture}
}
\put(60,0){
\begin{picture}(20,10)
\node[Nmarks=if,iangle=90,fangle=90](1)(0,5){$1$}\node(2)(20,5){$2$}
\drawedge[curvedepth=3](1,2){$a,b$}\drawedge[curvedepth=3](2,1){$a,b$}
\end{picture}
}
\end{picture}
\caption{A Stallings folding.}\label{figReversible}
\end{figure}
Since the automaton describes the group $\Z/2\Z$, we conclude that
the subgroup generated by $X$ is of index $2$ in the free group on $A$.
It is actually formed of the reduced words of even length.
\end{example}

\subsection{Strong, weak and neutral words}
Let $S$ be a factorial set.
For a word $w\in S$, let
\begin{displaymath}
m(w)=e(w)-\ell(w)-r(w)+1.
\end{displaymath}
We say that, with respect to
 $S$,  $w$ is \emph{strong} if $m(w)>0$,
\emph{weak} if $m(w)<0$ and \emph{neutral}  if $m(w)=0$.

A biextendable word $w$ is called \emph{ordinary} if $E(w)\subset a\times
A\cup A\times b$ for some $(a,b)\in E(w)$ (see~\cite{BertheRigo2010},
Chapter 4).
If $S$ is biextendable any ordinary word is neutral. Indeed, one has
$E(w)=(a\times (R(w)\setminus b))\cup ((L(w)\setminus a)\times b)\cup (a,b)$
and thus $e(w)=\ell(w)+r(w)-1$.

\begin{example}\label{exSturmianIsOrdinary}
 In a Sturmian set, any word is ordinary. Indeed, for
any bispecial word $w$, there is a unique letter $a$ such that
$aw$ is right-special and a unique letter $b$ such that $wb$
is left-special. Then $awb\in S$ and $E(w)=a\times A\cup A\times b$.
\end{example}
We say that a set $S$ is \emph{strong} (resp. \emph{weak},
resp. \emph{neutral}) if it is
factorial
and  every  word
$w\in S$ is strong or neutral (resp. weak or neutral, resp. neutral). 

The sequence $(p_n)_{n\ge 0}$ with $p_n=Card(S\cap A^n)$ is called the
\emph{factor complexity} (or complexity) of $S$. Set $k=\Card(S\cap A)-1$.
\begin{proposition}\label{propComplexityNeutral}
The factor complexity of a strong (resp. weak, resp. neutral) set $S$
is 
at least  (resp. at most, resp. exactly) equal to $kn+1$.
\end{proposition}

Given a factorial set $S$ with complexity $p_n$, we denote
$s_n=p_{n+1}-p_n$ the first difference of the sequence $p_n$
and $b_n=s_{n+1}-s_n$ its second difference.
The following is from~\cite{Cassaigne1997}
(it is also part of Theorem 4.5.4 in~\cite[Chapter 4]{BertheRigo2010}).

\begin{lemma}\label{lemmaEnum}
We have, for all $n\ge 0$,
\begin{displaymath}
b_n=\sum_{w\in A^n\cap S}m(w)\quad \text{ and } \quad s_n=\sum_{w\in A^n\cap
  S}(r(w)-1).
\end{displaymath}

\end{lemma}

Proposition~\ref{propComplexityNeutral} follows easily from
the fact that 
if $S$ is strong (resp. weak, resp. neutral), then $s_n\ge k$
(resp. $s_n\le k$, resp. $s_n=k$)
for all $n\ge 0$.

We now give an example of a set of complexity $2n+1$ on an alphabet
with three letters
which is not neutral.
\begin{example}\label{exampleChacon}
Let $A=\{a,b,c\}$.
The \emph{Chacon word} on three letters
is the fixpoint $x=f^\omega(a)$ of the morphism $f$ from
$A^*$ into itself defined by $f(a)=aabc$, $f(b)=bc$ and $f(c)=abc$.
Thus $x=aabcaabcbcabc\cdots$. The \emph{Chacon set} is the set $S$ of
factors of $x$. It is of complexity $2n+1$ (see~\cite{PytheasFogg2002}
Section 5.5.2).

It contains strong, neutral and weak words. Indeed,
$S\cap A^2=\{aa,ab,bc,ca,cb\}$ and thus $m(\varepsilon)=0$
showing that the empty word is neutral. Next
 $m(abc)=1$
and   $m(bca)=-1$, showing that $abc$ is strong while $bca$ is weak.
\end{example}

\subsection{Return words}

Let $S$ be a  set of words. For $w\in S$, let
\begin{displaymath}
\Gamma_S(w)=\{x\in S\mid wx\in S\cap A^+w\}\quad\text{and}\quad
\Gamma'_S(w)=\{x\in S\mid xw\in S\cap wA^+\}
\end{displaymath}
be respectively the set of \emph{right return words} and of
\emph{left return words} to $w$. If $S$ is recurrent, the sets
$\Gamma_S(w)$ and $\Gamma'_S(w)$ are nonempty. Let
\begin{displaymath}
\RR_S(w)=\Gamma_S(w)\setminus\Gamma_S(w) A^+\quad \text{and}\quad
 \RR'_S(w)=\Gamma'_S(w)\setminus A^+\Gamma'_S(w)
\end{displaymath}
be respectively the 
set of \emph{first right  return words} and the set of \emph{first
  left
return words} to $w$.  Note that $w\RR_S(w)=\RR'_S(w)w$.

Note that  a recurrent set $S$ is uniformly recurrent if and only if
the set $\RR_S(w)$ is finite for any $w\in S$. Indeed, if $N$
is the maximal length of the words in $\RR_S(w)$ for a word $w$
of length $n$, then two successive occurrences of $w$
in a word of $S$ are separated by a word of length at most $N-n$.
Thus any word in $S$ of length $N+n$ contains an occurrence
of $w$. The converse is obvious.

The following result has been proved  in~\cite{BalkovaPelantovaSteiner2008},
generalizing a property proved for Sturmian words 
in~\cite{JustinVuillon2000}
and for interval exchange in~\cite{Vuillon2007}.
\begin{theorem}\label{theoremCardReturn}
Let $S$ be a uniformly recurrent neutral set containing the alphabet $A$.
Then for every $w\in S$,
the set
 $\RR_S(w)$ has  $\Card(A)$ elements.
\end{theorem}
One can actually  prove more
generally,  for a uniformly recurrent set $S$, that
if $S$ is strong (resp. weak, resp. neutral), then for every $w\in S$,
the set
 $\RR_S(w)$ has at least (resp. at most, resp. exactly) $\Card(A)$ elements.

The following example shows that in a set of complexity $kn+1$
the number of first right return words need not be equal to $k+1$.
\begin{example}
Let $S$ be the Chacon set (see Example~\ref{exampleChacon}). We have
$\RR_S(a)=\{a,bca,bcbca\}$ but
$\RR_S(ab)=\{caab,cbcab\}$.
\end{example}

\section{Acyclic, connected and tree sets}\label{sectionAcyclic}
We  introduce in this section the notion of extension graph
of a word. We define acyclic (resp. connected, resp. tree)
sets by the fact that all the
extension graphs of its elements are acyclic (resp. connected,
resp. trees).
We give examples showing that a uniformly recurrent acyclic set
may not be a tree set (Example~\ref{exampleJulienAcyclic})
and that a uniformly recurrent neutral set may not be acyclic
(Example~\ref{exampleJulien}).
We introduce a generalization of the extension graphs
called generalized extension graphs.
 We give conditions under which generalized extension
graphs are acyclic
 (Proposition~\ref{PropStrongTreeCondition}). This allows
in particular to prove the closure under bifix decoding of
the family of acyclic sets, provided the result is biextendable
(Theorem~\ref{decodingAcyclic}).

\subsection{Extension graphs}
Let $S$ be a set of words.
For a  word $w\in S$, we
 consider an undirected graph $E_S(w)$ called its \emph{extension graph}
in $S$
and defined as follows. The set of vertices 
 is the disjoint union of
$L(w)$ and $R(w)$ and its edges are the pairs 
$(a,b)\in E(w)$. We also denote $E(w)$ instead of $E_S(w)$.

\begin{example}
Let $S$ be the Tribonacci set (see Example~\ref{exampleTribonacci}). The graphs $E(\varepsilon)$ and $E(ab)$
are represented in Figure~\ref{figureExtension}.

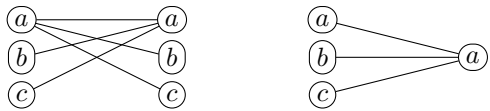
\begin{figure}[hbt]
\centering
\gasset{Nadjust=wh,AHnb=0}
\begin{picture}(70,10)
\put(0,0){
\begin{picture}(30,21)
\node(al)(0,10){$a$}\node(bl)(0,5){$b$}\node(cl)(0,0){$c$}
\node(ar)(20,10){$a$}\node(br)(20,5){$b$}\node(cr)(20,0){$c$}

\drawedge(al,ar){}\drawedge(al,br){}\drawedge(al,cr){}
\drawedge(bl,ar){}\drawedge(cl,ar){}
\end{picture}
}
\put(40,0){
\begin{picture}(30,10)
\node(al)(0,10){$a$}\node(bl)(0,5){$b$}\node(cl)(0,0){$c$}
\node(ar)(20,5){$a$}

\drawedge(al,ar){}
\drawedge(bl,ar){}\drawedge(cl,ar){}
\end{picture}
}
\end{picture}
\caption{The extension graphs $E(\varepsilon)$ and $E(ab)$ in the
Tribonacci set.}\label{figureExtension}
\end{figure}
\end{example}

We say that  $S$ is an \emph{acyclic} (resp. a connected,
resp. a tree) set
 if it is biextendable and if for every  word $w\in S$, the graph
 $E(w)$ is
acyclic (resp. connected, resp. a tree).
Obviously, a tree set is acyclic and connected.

Note  that a biextendable set $S$ is acyclic (resp. connected)
if and only if the graph $E(w)$ is acyclic (resp. connected)
for every bispecial word
$w$.
Indeed, if $w$ is not bispecial, then $E(w) \subset a \times A$ or
$E(w) \subset A  \times a$, thus it is always acyclic and connected. 

If the extension graph $E(w)$ of $w$ is acyclic, then 
$m(w)\le 0$.
Thus $w$ is weak or neutral. More precisely, one has 
in this case, $m(w)=-c+1$ where $c$ is the number
of connected components of the graph $E(w)$. 

Similarly, if $E(w)$
is connected, then $w$ is strong or neutral.
Thus, if $S$ is an acyclic (resp. a connected, resp. a tree) set,
then $S$ is a weak (resp. strong, resp. neutral) set.

\begin{example}
A Sturmian set $S$ is a tree set.
Indeed, any word $w\in S$ is ordinary
(Example~\ref{exSturmianIsOrdinary}),
which implies that $E(w)$ is a tree.
\end{example}

Since a tree set is neutral, we deduce from Proposition~\ref{propComplexityNeutral}
the following statement, where $k=\Card(S\cap A)-1$.
\begin{proposition}
The factor complexity of a tree set is $kn+1$.
\end{proposition}

One may wonder whether  the notion of a
 tree set   is  of  a topological or of a
  measure-theoretic nature
  for the associated symbolic dynamical system.
In particular, one may wonder if uniformly recurrent tree sets have the
property of unique ergodicity, which means that
they have a unique invariant probability measure
(see~\cite{BerstelDeFelicePerrinReutenauerRindone2012} or~\cite{BertheRigo2010}
for the definition of these notions).
An element of answer is provided   by interval exchange sets.

Regular interval exchange sets form a special case of uniformly
recurrent
tree sets (see \cite{BertheDeFeliceDolceLeroyPerrinReutenauerRindone2013}).
 It is well-known  since \cite{Keane1977}
 that  there exist  regular interval exchange sets that  are not
 uniquely  ergodic.  This  shows that  the tree property  does not  
imply  unique  ergodicity.
However  having  complexity  $p_n=kn+1$, which is  a priori of a
topological nature,
  implies  information on invariant measures. Indeed, according 
to~\cite{Boshernitzan1984},
a minimal symbolic dynamical  system  for which $ \liminf p_n/n \leq
k$   is  such
 that there exist at most   $k$ ergodic invariant  measures.
The bound can even be refined to $k-2$ \cite{Monteil2013} by a careful
inspection 
 of the evolution of the Rauzy graphs.
 For $k \leq 2$,  that is for an alphabet of size  at most $3$ in our
 case,  
 one  gets  the following~\cite{Boshernitzan1984}:
a   minimal symbolic        system  such that $ \limsup  p_n/n< 3$
 is uniquely ergodic.
We thus conclude that
any  uniformly recurrent  word  whose set of factors is a  
  tree set   on an alphabet
of size at most $3$ is uniquely ergodic.
\subsection{Two examples}\label{sectionExamples}
 We present 
two examples, due to Julien Cassaigne~\cite{Cassaigne2013}.
The first one is  a
uniformly recurrent  acyclic set which is not a tree set.
\begin{example}\label{exampleJulienAcyclic}
Let $A=\{a,b,c,d\}$ and let $\sigma$ be the morphism from $A^*$
into itself defined by
\begin{displaymath}
\sigma(a)=ab,\ \sigma(b)=cda,\ \sigma(c)=cd,\ \sigma(d)=abc.
\end{displaymath}
Let $S$ be the set of factors of the infinite word
$x=\sigma^\omega(a)$.
Since $\sigma$ is primitive, $S$ is uniformly recurrent. 
The graph $E(\varepsilon)$ is represented in
Figure~\ref{figureGepsilonJulien}.
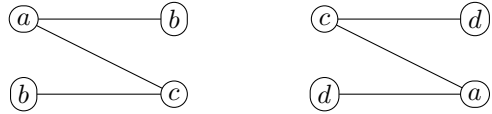
\begin{figure}[hbt]
\centering
\gasset{Nadjust=wh,AHnb=0}

\begin{picture}(70,10)
\node(a)(0,10){$a$}\node(a')(60,0){$a$}
\node(b)(0,0){$b$}\node(b')(20,10){$b$}
\node(c)(20,0){$c$}\node(c')(40,10){$c$}
\node(d)(40,0){$d$}\node(d')(60,10){$d$}

\drawedge(a,b'){}\drawedge(a,c){}
\drawedge(b,c){}
\drawedge(c',a'){}\drawedge(c',d'){}
\drawedge(d,a'){}
\end{picture}
\caption{The graph $E(\varepsilon$).}\label{figureGepsilonJulien}
\end{figure}
It is acyclic with two connected components (and thus $m(\varepsilon)=-1$).
We will show
that for any nonempty word
$w\in S$, the graph $E(w)$ is a tree. This will prove that $S$ is 
acyclic. Actually, let $\pi$ be the morphism from $A^*$ onto $\{a,b\}^*$
defined by $\pi(a)=\pi(c)=a$ and $\pi(c)=\pi(d)=b$. The image of $x$
by $\pi$ is the Sturmian word $y$ which is the fixpoint of the morphism
$\tau:a\mapsto ab,\ b\mapsto aba$. The word $x$ can be obtained back from
$y$ by changing one every other letter $a$  into a $c$ and any letter $b$
after a $c$ into a $d$. Thus every word of the set of factors $G$ of
$y$ gives rise to $2$ words in $S$.

In this way every bispecial word $w$ of $G$ gives two bispecial words
$w',w''$ of $S$ and their extension graphs in $S$ are isomorphic to $E_G(w)$.
For example, the word $ababa$ is bispecial in $G$. It gives the bispecial
words $abcda$ and $cdabc$ in $S$. Their extension graphs are shown below.
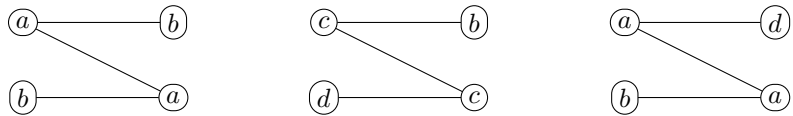
\begin{figure}[hbt]
\centering
\gasset{AHnb=0,Nadjust=wh}
\begin{picture}(100,10)
\put(0,0){
\node(b)(0,0){$b$}\node(a)(0,10){$a$}
\node(a')(20,0){$a$}\node(b')(20,10){$b$}

\drawedge(a,b'){}\drawedge(a,a'){}
\drawedge(b,a'){}
}
\put(40,0){
\node(d)(0,0){$d$}\node(c)(0,10){$c$}
\node(c')(20,0){$c$}\node(b')(20,10){$b$}

\drawedge(c,b'){}\drawedge(c,c'){}
\drawedge(d,c'){}
}
\put(80,0){
\node(b)(0,0){$b$}\node(a)(0,10){$a$}
\node(a')(20,0){$a$}\node(d')(20,10){$d$}

\drawedge(a,d'){}\drawedge(a,a'){}
\drawedge(b,a'){}
}
\end{picture}
\caption{The  graphs  $E_G(ababa)$,
$E_S(abcda)$ and $E_S(cdabc)$.}
\end{figure}

This proves that $S$ is acyclic.
\end{example}

The second example is a uniformly recurrent set which is neutral
but is not a tree set (it is actually not even acyclic).
\begin{example}\label{exampleJulien}

Let $B=\{1,2,3\}$ and let $\tau:A^*\rightarrow B^*$ be defined by
\begin{displaymath}
\tau(a)=12,\quad \tau(b)=2,\quad \tau(c)=3,\quad \tau(d)=13.
\end{displaymath}
Let $G=\tau(S)$ where $S$ is the set of
Example~\ref{exampleJulienAcyclic}.
Thus $G$ is also
 the set of factors of the infinite word $\tau(\sigma^\omega(a))$.

The set $Y=\tau(A)$ is a prefix code. It is not a suffix code but it
is \emph{weakly suffix} in the sense that if $x,y,y'\in X$
and $x'\in X^*$ are such that $xy$ is a suffix of $x'y'$, then $y=y'$.

Let $g:\{a,c\}A^*\cap A^*\{a,c\}\rightarrow B^*$ be the map defined by
\begin{displaymath}
g(w)=\begin{cases}3\tau(w)&\text{if $w$ begins and ends with $a$}\\
3\tau(w)1&\text{if $w$ begins with $a$ and ends with $c$}\\
2\tau(w)&\text{if $w$ begins with $c$ and ends with $a$}\\
2\tau(w)1&\text{if $w$ begins with $c$ and ends with $c$}\\
\end{cases}
\end{displaymath}
It can be verified, using the fact that $Y$
is a prefix and weakly suffix code, that the set of nonempty bispecial words of
$G$ is the
union of $2$, $31$ and of the
set $g(S)$ where $S$ is the set of nonempty bispecial words of $S$.
One may verify that the words of $g(S)$ are neutral.
 Since the  words $2$, $31$ are also neutral,
the set $G$ is neutral. 

It is uniformly recurrent since $S$ is uniformly recurrent
and $\tau$ is a nontrivial morphism. The set $G$ is not a tree set
since the graph $E(\varepsilon)$ is neither
acyclic nor  connected (see Figure~\ref{GepsilonJC}).
\begin{figure}[hbt]
\centering\gasset{Nadjust=wh,AHnb=0}
\begin{picture}(20,10)
\node(a1)(0,10){$1$}\node(a2)(20,0){$1$}
\node(b1)(0,5){$2$}\node(b2)(20,10){$2$}
\node(c1)(0,0){$3$}\node(c2)(20,5){$3$}

\drawedge(a1,b2){}\drawedge(a1,c2){}
\drawedge(b1,b2){}\drawedge(b1,c2){}
\drawedge(c1,a2){}
\end{picture}
\caption{The graph $E(\varepsilon)$ for the set $G$.}\label{GepsilonJC}
\end{figure}
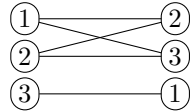

\end{example}

\subsection{Generalized extension graphs}
Let $S$ be a set. For $w\in S$, and $U,V\subset S$,
let 
$U(w)=\{\ell\in U\mid \ell w\in S\}$
 and let $V(w)=\{r\in V\mid wr\in S\}$.
The \emph{generalized extension graph} of $w$ relative to
$U,V$ is the following undirected graph $E_{U,V}(w)$. The set of vertices is
made of two disjoint copies of $U(w)$ and $V(w)$.
 The edges are the pairs $(\ell,r)$
for $\ell\in U(w)$ and $r\in V(w)$
such that $\ell wr\in S$. The extension graph $E(w)$ defined previously
corresponds
to the case where $U,V=A$.

\begin{example}
Let $S$ be the Fibonacci set. Let $w=a$, $U=\{aa,ba,b\}$ and let
$V=\{aa,ab,b\}$. The graph $E_{U,V}(w)$ is represented in
Figure~\ref{figureStrongTree}.
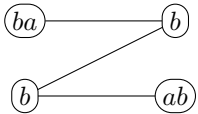
\begin{figure}[hbt]
\centering
\gasset{AHnb=0, Nadjust=wh}
\begin{picture}(20,10)
\node(b)(0,0){$b$}\node(ba)(0,10){$ba$}
\node(ab)(20,0){$ab$}\node(b')(20,10){$b$}

\drawedge(ba,b'){}\drawedge(b,b'){}\drawedge(b,ab){}
\end{picture}
\caption{The graph $E_{U,V}(w)$.}\label{figureStrongTree}
\end{figure}
\end{example}

The following property shows that in an acyclic set, not only
the extension graphs but, under appropriate hypotheses,
 all generalized extension graphs
are acyclic.
\begin{proposition}\label{PropStrongTreeCondition}
Let $S$ be an acyclic set.  For any $w\in S$, any
 finite 
suffix code  $U$ and any  finite  prefix code $V$,
 the generalized extension graph $E_{U,V}(w)$ is  acyclic.
\end{proposition}
The proof uses the following lemma.
\begin{lemma}\label{lemmaTree}
Let $S$ be a  biextendable set. Let $w\in S$ and let $U,V,T\subset S$.
Let
 $\ell\in S\setminus U$  be such that $\ell w\in S$. Set
 $U'=(U\setminus T\ell
)\cup\ell$. If the graphs $E_{U',V}(w)$
and $E_{T,V}(\ell w)$
are acyclic   then $E_{U,V}(w)$ is  acyclic.
\end{lemma}
\begin{proof}
 Assume  that
$E_{U,V}(w)$ contains a cycle $C$. If the cycle does not use a vertex
in $U'$, it defines a cycle in the graph $E_{T,V}(\ell w)$
obtained by replacing each vertex $t\ell$ for $t\in T$ by a vertex $t$.
Since $E_{T,V}(\ell w)$ is acyclic, this is impossible.
If it uses a vertex of $U'$ it defines a cycle of
the graph $E_{U',V}(w)$ obtained by replacing each possible vertex $t\ell$
by $\ell$ (and suppressing the possible identical
successive edges created by the identification). This is impossible since $E_{U',V}(w)$ is acyclic.
Thus $E_{U,V}(w)$ is  acyclic.
\end{proof}

\begin{proofof}{of Proposition~\ref{PropStrongTreeCondition}}
We show by induction on the sum of the lengths of the words in $U,V$
that for any $w\in S$, the graph $E_{U,V}(w)$ is  acyclic.

Let $w\in S$. We may assume that $U=U(w)$ and $V=V(w)$
and also that $U,V\ne\emptyset$.
If $U,V\subset A$, the property is true since $S$ is acyclic. 

Otherwise, assume for example that $U$ contains words of length at
least $2$. 
Let $u\in U$ be of maximal length. Set $u=a\ell$ with $a\in A$.
Let $T=\{b\in A\mid b\ell\in U\}$.
 Then $U'=(U\setminus T\ell
)\cup\ell$ is a suffix code and $\ell w\in S$ since $U=U(w)$. 

By induction hypothesis, the graphs $E_{U',V}(w)$
and $E_{T,V}(\ell w)$
are acyclic. 
By lemma~\ref{lemmaTree}, the graph  $E_{U,V}(w)$ is acyclic.
\end{proofof}
We prove now a similar statement concerning tree sets. 
\begin{proposition}\label{propStrongTreeConditionBis}
Let $S$ be a tree set.  For any $w\in S$, any
 finite 
$S$-maximal suffix code  $U\subset S$ 
and any  finite  $S$-maximal prefix code $V\subset S$,
 the generalized extension graph $E_{U,V}(w)$ is  a tree.
\end{proposition}
The proof uses the following lemma, analogous to Lemma~\ref{lemmaTree}.
\begin{lemma}\label{lemmaTreeBis}
Let $S$ be a  biextendable set. Let $w\in S$ and let $U,V\subset S$.
Let
 $\ell\in S\setminus U$  be such that $\ell w\in S$ and $A\ell\cap
S\subset U$. Set
 $U'=(U\setminus A\ell
)\cup\ell$. If the graphs $E_{U',V}(w)$
and $E_{A,V}(\ell w)$
are connected   then $E_{U,V}(w)$ is  connected.
\end{lemma}
\begin{proof}
Since $S$ is left extendable, there is a letter $a$ such that $a\ell
w\in S$ and thus $a\ell\in U(w)$.
We proceed by steps.

Step 1. As a preliminary step, let us show that for each $b\in A$
such that $b\ell w\in S$,
 and each  $v\in V(\ell w)$,
there is a path from $b\ell$ to $v$
in  $E_{U,V}(w)$. Indeed,  since the graph $E_{A,V}(\ell
w)$
is connected there is a path from $b$ to $v$ in this graph.
Thus, since $b\ell\in U(w)$, there is a path from $b\ell$ to $v$ in 
$E_{U,V}(w)$.

Step 2. As a second step, let us show that for any 
$m\in U'(w)\setminus\ell$ and $v\in V(w)$,
there is a path from $m$ to $v$ in $E_{U,V}(w)$.
Indeed there is a path from $m$ to $v$ in $E_{U',V}(w)$.
For each  edge of this
path of the form $(\ell,s)$, $s$ is also in $V(\ell w)$ and thus, by Step 1, 
there is a path from $a\ell$ to $s$
in the graph $E_{U,V}(w)$. Thus there is a path from $m$ to $v$
in $E_{U,V}(w)$.

Step 3. For each $b \in A$
such that $b\ell\in U(w)$, for each $v \in V(w)$,
there is a path from $b\ell$ to $v$ in $E_{U,V}(w)$.
 Indeed, since $E_{A,V}(\ell w)$
is connected, there is a path from $b$ to $a$ in $E_{A,V}(\ell w)$,
 thus a path from $b\ell$ to $a\ell$ in 
$E_{U,V}(w)$. Then there is a path
 from $\ell$ to $v$ in $E_{U',V}(w)$ and, in the same way as in Step 2, there 
 is a path from
$a\ell$ to $v$ in $E_{U,V}(w)$.

Consider now $m\in U(w)$ and $v\in V(w)$. If $m\notin A\ell$, then $m\in U'(w)\setminus\ell$
and thus, by Step 2, there is a path from $m$ to $v$
in $E_{U,V}(w)$. Next, assume that  $m=b\ell$ with $b\in A$.
By Step 3, there is a path from $m$ to $v$ in
$E_{U,V}(w)$. This shows that the graph $E_{U,V}(w)$ is connected.
\end{proof}

\begin{proofof}{of Proposition~\ref{propStrongTreeConditionBis}}
The fact that $E_{U,V}(w)$ is acyclic follows from 
Proposition~\ref{PropStrongTreeCondition}.

We show by induction on the sum of the lengths of the words in $U,V$
that for any $w\in S$, the graph $E_{U,V}(w)$ is connected.

Assume first that $U(w),V(w)\subset A$. Since $U$ is an $S$-maximal
suffix code, we have $U(w)=L(w)$. Similarly, $V(w)=R(w)$. Thus
the property is true since $S$ is a tree set. 

Otherwise, assume for example that $U(w)$ contains words of length at
least $2$. 
Let $u\in U(w)$ be of maximal length. Set $u=a\ell$ with $a\in A$.
 Then $U'=(U\setminus A\ell
)\cup\ell$ is an $S$-maximal suffix code and $\ell w\in S$ since
 $a\ell\in U(w)$. 
Moreover, we have $A\ell \cap S\subset U$ since $U$ is an $S$-maximal
suffix code.
Thus $\ell$ satisfies the hypotheses of Lemma~\ref{lemmaTreeBis}.

By induction hypothesis, the graphs $E_{U',V}(w)$
and $E_{A,V}(\ell w)$
are connected. 
By Lemma~\ref{lemmaTreeBis}, the graph  $E_{U,V}(w)$ is connected.
\end{proofof}

Let $S$ be a factorial set and let $f$ be a coding morphism for a
finite  bifix code $X\subset S$. The set $f^{-1}(S)$ is called
a \emph{bifix decoding} of $S$. When $X$ is an $S$-maximal bifix code,
it is called a \emph{maximal bifix decoding} of $S$.

\begin{theorem}\label{decodingAcyclic}
Any biextendable set which is the bifix decoding of an acyclic set is acyclic.
\end{theorem}
\begin{proof}
Let $S$ be an acyclic set and let $f:B^*\rightarrow A^*$ be a coding
morphism for a finite bifix code $X\subset S$ such that
$f^{-1}(S)$ is biextendable.
Let $u\in f^{-1}(S)$ and let $v=f(u)$. 
 Since $X$
is a finite bifix code, it is both a
suffix code and a prefix code. Thus the generalized extension 
graph $E_{X,X}(v)$ is acyclic  by
Proposition~\ref{PropStrongTreeCondition}. Since $E(u)$ is isomorphic
with $E_{X,X}(v)$, it is also acyclic. Thus $f^{-1}(S)$ is acyclic.
\end{proof}
The previous statement is not satisfactory because of the
assumption that $f^{-1}(S)$ is biextendable which is added to
obtain the conclusion. The following example shows that
the condition is necessary.
\begin{example}\label{exampleNotExtendable}
Let $S$ be the Fibonacci set and let $f$ be the coding morphism
for $X=\{aa,ab\}$ defined by $f(u)=aa$, $f(v)=ab$. Then 
$f^{-1}(S)$ is the finite set $\{u,v,vu,vv,vvu\}$
and thus not biextendable. Note however that for any 
$w\in f^{-1}(S)$, the graph $E(w)$ is acyclic.
\end{example}
One may verify that a sufficient condition for $f^{-1}(S)$
to be biextendable is that  $X$ is an $S$-maximal
prefix code and an $S$-maximal suffix code (when $S$ is recurrent,
this is equivalent to the fact that $X$ is an $S$-maximal bifix code).

The  following result is a consequence of
Proposition~\ref{propStrongTreeConditionBis}.
\begin{theorem}\label{InverseImageTree}
Any maximal bifix decoding of a recurrent tree set is a tree set.
\end{theorem}
\begin{proof}  Let $f:B\rightarrow X$
 be a coding morphism for a finite $S$-maximal bifix code $X$.
Since $S$ is recurrent, it is
biextendable. It implies that $f^{-1}(S)$ is also biextendable. Indeed,
let $u\in f^{-1}(S)$ and let $v=f(u)$. Let $r,s$ be words of $S$
longer than all words of $X$ such that $rvs\in S$. Let
$r'$ (resp. $s'$) be the suffix of $r$ (resp. the prefix of $s$) which
is in $X$. Then $f^{-1}(r')uf^{-1}(s')$ is in $f^{-1}(S)$.
This shows that $f^{-1}(S)$ is biextendable.

Let $u\in f^{-1}(S)$ and let $v=f(u)$.
Since $S$ is a tree set,
it satisfies 
Proposition~\ref{propStrongTreeConditionBis}. Since $S$ is recurrent
and $X$
is a finite $S$-maximal bifix code, $X$ is both an $S$-maximal
suffix code and an $S$-maximal prefix code. Thus the
graph $E_{X,X}(v)$ is a tree. Since $E(u)$ is isomorphic
with $E_{X,X}(v)$, it is also a tree. Thus $f^{-1}(S)$ is a tree set.
\end{proof}
We have no example of a maximal bifix decoding of a recurrent tree
set which is not recurrent.
\begin{example}\label{exampleF2}
Let $S$ be the Fibonacci set and let $X=A^2\cap S=\{aa,ab,ba\}$.
Let $B=\{u,v,w\}$ and let
$f$ be the coding morphism for $X$ defined by $f(u)=aa$,
$f(v)=ab$ and $f(w)=ba$. Then the set $f^{-1}(S)$ is a recurrent tree set
which is actually a regular interval exchange set 
(see~\cite{BertheDeFeliceDolceLeroyPerrinReutenauerRindone2013}).

\end{example}

\section{Return words in tree sets}\label{sectionReturnTreeSets}
We study sets of first return words in tree sets. We first show
that if $S$ is a recurrent connected set, the group described by
any Rauzy graph of $S$ containing the alphabet $A$,
with respect to some vertex is the free group
on $A$ (Theorem~\ref{proposition3}). Next, we prove that
in a uniformly recurrent tree set containing $A$, the set of first return words
to any word of $S$ is a basis of the free group on $A$ (Theorem~\ref{theoremJulien}).
\subsection{Stallings foldings of Rauzy graphs}\label{sectionRauzyGraphs}
We first introduce the notion of a Rauzy graph (for a more detailed
exposition, see~\cite{BertheRigo2010}).
Let $S$ be a factorial set.
The \emph{Rauzy graph} of $S$ of order $n\ge 0$ is the following labeled
graph $G_n(S)$. Its vertices are the words in the set $S\cap A^n$.
Its edges are the triples $(x,a,y)$ for all $x,y\in S\cap A^n$
and $a\in A$ such that $xa\in S\cap Ay$.

Let $u\in S\cap A^n$. The following properties follow easily from the definition
of the Rauzy graph.
\begin{enumerate}
\item[(i)] For any word $w$ such that $uw\in S$,
there is a path  labeled $w$ in $G_n(S)$ from $u$ to the suffix
of length $n$ of $uw$.

\item[(ii)] Conversely, the label of any path of length at most $n+1$ in $G_n(S)$
is in $S$.
\end{enumerate}

When $S$ is recurrent, all Rauzy graph $G_n(S)$
are strongly connected. Indeed, let $u,w\in S\cap A^n$.
Since $S$ is recurrent, there is a $v\in S$ such
that $uvw\in S$. Then there is a path
in $G_n(S)$ from $u$ to $w$ labeled $vw$ by property (i) above.

The Rauzy graph $G_n(S)$ of a recurrent set $S$ with a distinguished
vertex $v$ can be considered as a simple automaton $\A=(Q,v,v)$ with
set of states $Q=S\cap A^n$ (see Section~\ref{sectionAutomata}).

Let $G$ be a labeled graph on a set $Q$ of vertices.
The group described by $G$ with respect to a vertex $v$
is the subgroup described by the simple automaton $(Q,v,v)$.
 We will prove the following statement.
\begin{theorem}\label{proposition3}
Let $S$ be a recurrent connected set containing the
alphabet $A$.
The group described by a Rauzy graph of $S$
with respect to any vertex is the free group on $A$.
\end{theorem}

A \emph{morphism} $\varphi$ from a labeled graph $G$ onto a labeled graph
$H$ is a map from the set of vertices of $G$ onto the set
of vertices of $H$ such that $(u,a,v)$ is an edge of $H$ if and only
if there is an edge $(p,a,q)$ of $G$ such that $\varphi(p)=u$
and $\varphi(q)=v$. An \emph{isomorphism} of labeled graphs
 is a bijective morphism.

The \emph{quotient} of a labeled graph $G$ by an equivalence $\theta$,
denoted $G/\theta$,
is the graph with vertices the set of equivalence classes of $\theta$
and an edge from the class of $u$ to the class of $v$
labeled $a$ if there is an edge labeled $a$ from a vertex $u'$
equivalent to $u$ to a vertex $v'$ equivalent to $v$.
The map from a vertex of $G$ to its equivalence class is
a morphism from $G$ onto $G/\theta$.

We consider on a Rauzy graph $G_n(S)$ the equivalence $\theta_n$
formed by the pairs
$(u,v)$ with  $u=ax$, $v=bx$, $a,b\in L(x)$
such that there is
a path from $a$ to $b$ in the extension graph $E(x)$ (and more precisely
from the vertex corresponding to $a$ to the vertex corresponding
to $b$ in the copy corresponding to $L(x)$ in the bipartite graph
$E(x)$).
\begin{proposition}\label{propRauzyGraphs}
If $S$ is connected, for each $n\ge 1$, the quotient
of $G_n(S)$ by the equivalence $\theta_n$ is isomorphic to $G_{n-1}(S)$.
\end{proposition}
\begin{proof}
The map $\varphi:S\cap A^n\rightarrow S\cap A^{n-1}$ mapping a word of
$S$ of length $n$
to its suffix of length $n-1$ is clearly a morphism from $G_n(S)$
onto $G_{n-1}(S)$. If $u,v\in S\cap A^n$ are equivalent modulo $\theta_n$,
then $\varphi(u)=\varphi(v)$. Thus there is a morphism $\psi$ from
$G_n(S)/\theta_n$ onto $G_{n-1}(S)$. It is defined for any word $u\in S\cap
A^n$ by
$\psi(\bar{u})=\varphi(u)$
where $\bar{u}$ denotes the class of $u$ modulo $\theta_n$.
But since $S$ is connected,
the class modulo $\theta_n$
of a word $ax$ of length $n$ has $\ell(x)$ elements, which is the same
as
the number of elements of $\varphi^{-1}(x)$.
This shows that $\psi$ is a surjective map
from a finite set onto a set of the same cardinality
and thus that it is one-to-one. Thus $\psi$ is an isomorphism.
\end{proof}
Let $G$ be a strongly connected labeled graph.
Recall from Section~\ref{sectionAutomata}
that a Stallings folding at vertex $v$ relative to letter $a$ of $G$
consists in 
identifying the edges coming into $v$ labeled $a$ and identifying
their origins. A Stallings folding does not modify the group described
by the graph with respect to some vertex. Indeed, if $p\edge{a}v$, $p\edge{b}r$ and $q\edge{a}v$
are three edges of $G$, then adding the edge $q\edge{b}r$ does
not change the group described since the path
$q\edge{a}v\edge{a^{-1}}p\edge{b}r$
has the same label. Thus merging $p$ and $q$ does not add new labels
of
generalized paths.\\

\begin{proofof}{of Theorem~\ref{proposition3}}
The quotient $G_n(S)/\theta_n$ can be obtained by a sequence of Stallings
foldings from the graph $G_n(S)$. Indeed, a Stallings folding at vertex $v$ identifies
vertices
which are equivalent modulo $\theta_n$.
Conversely, consider $u=ax$ and $v=bx$, with $u,v\in S\cap A^n$
and $a,b\in A$ such that
$a$ and $b$ (considered as elements of $L(x)$),
are connected by a path in $E(x)$. Let $a_0,\ldots a_k$ and
$b_1,\cdots b_{k}$ with $a=a_0$ and $b=a_k$ be such that
$(a_i,b_{i+1})$ for $0\le i\le k-1$
 and $(a_i,b_i)$ for $1\le i\le k$ are in $E(x)$. The successive
Stallings foldings at $xb_1,xb_2,\ldots,xb_k$ identify the
vertices $u=a_0x,a_1x,\ldots,a_kx=v$.
Indeed, since $a_ixb_{i+1}, a_{i+1}xb_{i+1} \in S$,
there are two edges labeled  $b_{i+1}$
going out of $a_ix$ and $a_{i+1}x$ which end at
$xb_{i+1}$. The Stallings folding identifies
 $a_ix$ and $a_{i+1}x$. The conclusion follows by induction.

Since the Stallings foldings do not modify the group described, we
deduce
from Proposition~\ref{propRauzyGraphs} that the group described
by the Rauzy graph $G_n(S)$ is the same as the group described by
the Rauzy graph $G_0(S)$. Since $G_0(S)$ is the graph with
one vertex and with loops labeled by each of the letters,
it describes the free group on $A$.
\end{proofof}
\begin{example}
Let $S$ be the tree set obtained by decoding the Fibonacci set
into blocks of length $2$ (see Example~\ref{exampleF2}). Set $u=aa$, $v=ab$, $w=ba$. The graph
$G_2(S)$ is represented on the left of Figure~\ref{figFiboBlocks}.
\begin{figure}[hbt]
\centering
\gasset{Nadjust=wh}
\begin{picture}(100,20)
\node(wv)(0,10){$wv$}\node(vv)(15,10){$vv$}\node(vu)(30,10){$vu$}\node(uw)(45,10){$uw$}
\node(ww)(20,0){$ww$}
\drawedge(wv,vv){$v$}\drawedge(vv,vu){$u$}\drawedge(vu,uw){$w$}
\drawedge[curvedepth=-7,ELside=r](uw,wv){$v$}\drawedge[curvedepth=5](uw,ww){$w$}
\drawedge[curvedepth=5](ww,wv){$v$}
\drawedge[curvedepth=-5,ELside=r](wv,vu){$u$}

\node(v)(60,10){$v$}\node(u)(75,10){$u$}\node(w)(90,10){$w$}
\drawloop[loopangle=-90](v){$v$}\drawedge(v,u){$u$}\drawedge(u,w){$w$}
\drawloop[loopangle=-90](w){$w$}\drawedge[curvedepth=-7,ELside=r](w,v){$v$}
\end{picture}
\caption{The Rauzy graphs $G_2(S)$ and $G_1(S)$ for the decoding of the Fibonacci set into blocks of length $2$.}\label{figFiboBlocks}
\end{figure}
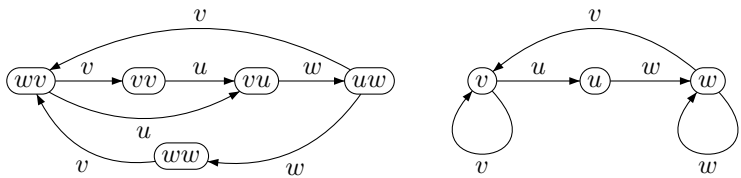
The classes of $\theta_2$ are $\{wv,vv\}$ $\{vu\}$ and $\{ww,uw\}$.
The graph $G_1(S)$
is represented on the right.
\end{example}
The following example shows that Proposition~\ref{propRauzyGraphs} is false for sets
which are not connected.
\begin{example}
Consider again the Chacon set (see Example~\ref{exampleChacon}).

The Rauzy graph $G_1(S)$ corresponding to the Chacon set is
represented in Figure~\ref{figChacon2} on the left.
The graph $G_1(S)/\theta_1$ is represented on the right. It is
not isomorphic to $G_0(S)$ since it has two vertices instead of one.
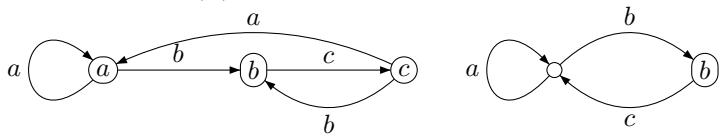
\begin{figure}[hbt]
\centering
\gasset{Nadjust=wh}
\begin{picture}(80,10)
\node(a)(0,5){$a$}\node(b)(20,5){$b$}\node(c)(40,5){$c$}

\drawloop[loopangle=180](a){$a$}\drawedge(a,b){$b$}
\drawedge(b,c){$c$}\drawedge[curvedepth=-5,ELside=r](c,a){$a$}\drawedge[curvedepth=5](c,b){$b$}

\node(ac)(60,5){}\node(b')(80,5){$b$}
\drawloop[loopangle=180](ac){$a$}\drawedge[curvedepth=5](ac,b'){$b$}
\drawedge[curvedepth=5](b',ac){$c$}
\end{picture}
\caption{The graphs $G_1(S)$ and $G_1(S)/\theta_1$.}\label{figChacon2}
\end{figure}
\end{example}

\subsection{The Return Theorem}
We will prove the following result (referred to as the Return Theorem).
\begin{theorem}\label{corollaryJulien}
Let $S$ be a uniformly recurrent tree set containing the alphabet $A$. Then for any $w\in S$, the
set
$\RR_S(w)$ is a basis of the free group on $A$.
\end{theorem}
We first show an example of a neutral set which is not a tree set
and for which Theorem~\ref{corollaryJulien} does not hold.
\begin{example}
Consider the set $S$ of Example~\ref{exampleJulien}. Then
$\RR_S(1)=\{2231,31,231\}$. This set has $3$ elements, in agreement
with Theorem~\ref{theoremCardReturn} but it is not a basis of the free
group on $\{1,2,3\}$ since it generates the same group as
$\{2,31\}$.
\end{example}

The proof of Theorem~\ref{corollaryJulien} uses Theorem~\ref{theoremCardReturn}
and the following result.

\begin{theorem}\label{theoremJulien}
Let $S$ be a uniformly 
recurrent connected set containing the alphabet $A$. For any $w\in S$, the set
$\RR_S(w)$ generates the free group on $A$.
\end{theorem}
\begin{proof}
Since $S$ is uniformly recurrent, the set $\RR_S(w)$ is finite. Let $n$
be the maximal length of the words in $w\RR_S(w)$. In this way, any word
in $S\cap A^n$ beginning with $w$ has a prefix in $wR_S(w)$.
Moreover, recall from Property (ii) of Rauzy graphs,
 that the label of any path of
length $n+1$ in the Rauzy graph $G_n(S)$ is in $S$. 

Let $x\in S$ be a word of
length $n$ ending with $w$. Let $\A$ be the simple automaton defined by $G_n(S)$ with initial and
terminal state $x$. Let $X$ be the prefix code generating the
submonoid
recognized by $\A$. Since the automaton $\A$ is simple, by
Proposition~\ref{propGeneratedGroup}, the
set $X$ generates the group described by $\A$. 

We show that $X\subset \RR_S(w)^*$. 
Indeed, let
$y\in X$. Since $y$ is the label
of a path starting at $x$ and ending in $x$, the word $xy$ ends with
$x$
and thus the word $wy$ ends with $w$. Let $\Gamma=\{z\in A^+\mid wz\in
A^*w\}$
and let $R=\Gamma\setminus \Gamma A^+$. Then $R$ is a prefix code
and $\Gamma\cup 1=R^*$, as one may verify easily.
Since $y\in\Gamma$,
 we can write
$y=u_1u_2\cdots u_m$ where each word $u_i$ is in $R$. 
Since $S$ is recurrent and since
$x\in S$, there is $v\in S\cap A^n$ such that $vx\in S$
and thus there is a path labeled $x$ ending at the vertex $x$
by property (i) of Rauzy graphs.
Thus there is a path labeled $xy$ in $G_n(S)$. This implies that
for $1\le i\le m$, there is a path in $G_n(S)$ labeled $wu_i$.

Assume that some $u_i$ is such that
$|wu_i|>n$. 
Then the prefix $p$ of length $n$  of $wu_i$ is the
label of
a path in $G_n(S)$. This implies, 
by Property (ii) of Rauzy graphs,
 that $p$ is in $S$ and thus that $p$
 has a prefix in $wR_S(w)$. But then $wu_i$ has
a proper prefix in $wR_S(w)$, a contradiction. Thus we
have
$|wu_i|\le n$ for all $i=1,2,\ldots,m$. But then the $wu_i$ are in $S$
by property (i) again and
thus the $u_i$ are in $\RR_S(w)$. This shows that $y\in \RR_S(w)^*$.

Thus the group generated by $\RR_S(w)$ contains the group generated by $X$.
But, by Theorem~\ref{proposition3}, the group described by $\A$
is the free group on $A$. Thus $\RR_S(w)$ generates the free group on $A$.
\end{proof}
We illustrate the proof in the following example.
\begin{example}
Let $S$ be the Fibonacci set. We have $\RR_S(aa)=\{baa,babaa\}$. The
Rauzy
graph $G_7(S)$ is represented in Figure~\ref{figureRauzyGraphG_7}.
The set recognized by the automaton obtained using $x=aababaa$ as
initial and terminal state is $X^*$ with
$X=\{babaa,baababaa\}$. In agreement
with the proof of Theorem~\ref{theoremJulien}, we have $X\subset \RR_S(aa)^*$.
\begin{figure}[hbt]
\centering\gasset{Nadjust=wh}
\begin{picture}(80,30)
\node(4)(0,20){$abaabab$}\node(7)(20,20){$baababa$}
\node[Nmarks=if,iangle=-90,fangle=-90](2)(40,20){$aababaa$}\node(5)(60,20){$ababaab$}
\node(8)(80,20){$babaaba$}\node(3)(70,10){$abaabaa$}
\node(6)(40,0){$baabaab$}\node(1)(10,10){$aabaaba$}

\drawedge[curvedepth=3](1,4){$b$}
\drawedge(4,7){$a$}\drawedge(7,2){$a$}
\drawedge(2,5){$b$}\drawedge(5,8){$a$}
\drawedge[curvedepth=3](8,3){$a$}\drawedge[curvedepth=-10](8,4){$b$}
\drawedge[curvedepth=3](3,6){$b$}\drawedge[curvedepth=3](6,1){$a$}

\end{picture}
\caption{The Rauzy graph $G_7(S)$}\label{figureRauzyGraphG_7}
\end{figure}
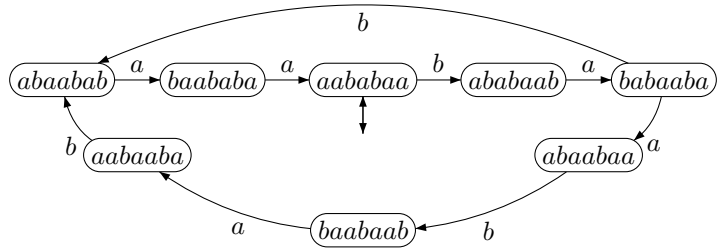
\end{example}

\begin{proofof}{of Theorem~\ref{corollaryJulien}}
 When
$S$ is a tree set, we have 
$\Card(\RR_S(w))=
\Card(A)$ by Theorem~\ref{theoremCardReturn}, which implies the conclusion
since any  set with $\Card(A)$ elements generating $F_A$
is a basis of $F_A$.
\end{proofof}

\section{Bifix codes in acyclic sets}\label{sectionMainResult}
We prove in this section our main results.  Bifix codes in acyclic sets
are bases of the subgroup that they generate
(Theorem~\ref{basisTheorem},
referred to as the Freeness Theorem). Moreover, the
submonoid generated by a finite bifix code $X$ included in
an acyclic set $S$ is such that $X^*\cap S=\langle X\rangle\cap S$
(Theorem~\ref{saturationTheorem}, referred to as the Saturation Theorem).
As a preliminary to the proof, we first define the incidence graph of a finite bifix code (already used
in~\cite{BerstelDeFelicePerrinReutenauerRindone2012}). We prove
a result concerning this graph, implying in particular that
it is acyclic (Proposition~\ref{newLemma633}). We then define
the coset automaton whose states are connected components of
the incidence graph. We prove that this automaton is the Stallings
automaton of the subgroup $\langle X\rangle$
(Proposition~\ref{lemmaBidet}).
Finally, we prove the Freeness and the Saturation Theorems.
\subsection{Freeness and Saturation Theorems}
Let $X$ be a subset of the free group.
We say that $X$ is \emph{free} if
it is a basis of the subgroup $\langle X\rangle$ generated by $X$.
This means that if $x_1,x_2,\ldots,x_n\in X\cup X^{-1}$ 
are such
that $x_1x_2\cdots x_n$ is equivalent to $1$, then $x_ix_{i+1}$
is equivalent to $1$ for some $1\le i<n$.

We will prove the following result (Freeness Theorem).
\begin{theorem}\label{basisTheorem}
A set  $S$ is acyclic if and only if
 any  bifix code  $X\subset S$
is a free subset of the free group $F_A$.
\end{theorem}

Let $M$ be a submonoid of $A^*$ and let $H$ be the subgroup
of $F_A$ generated by $M$. Given a set of words $S$, the submonoid $M$
is said to be \emph{saturated} in $S$ if
$M\cap S=H\cap S$. If $M$ is  generated by $X$,
 then $M$ is saturated in $S$ if and only if $X^*\cap S=\langle X\rangle\cap S$.

Thus, for example, the submonoid recognized by a reversible 
automaton is saturated in $A^*$ (Proposition~\ref{propGeneratedGroup}).

We will prove the
following result (Saturation Theorem).
\begin{theorem}\label{saturationTheorem}
Let $S$ be an acyclic set. The submonoid generated by a bifix code
included in $S$ is saturated in $S$.
\end{theorem}

We note the following corollary, which shows that
bifix codes in acyclic sets satisfy a property
which is stronger than being bifix (or more precisely that
the submonoid $X^*$ satisfies a property stronger
than being right and left unitary).
\begin{corollary}\label{corollaryChristophe}
Let $S$ be an acyclic set, let $X\subset S$ be a bifix code
and let $H=\langle X\rangle$. For any $u,v\in S$,
\begin{enumerate}
\item[\rm (i)] if
$u,uv\in H\cap S$, then $v\in X^*$,
\item[\rm (ii)] if $v,uv\in H\cap S$, then $u\in X^*$.
\end{enumerate}
\end{corollary}
\begin{proof}
Assume that $u,uv\in H\cap S$. Since $v\equiv u^{-1}(uv)$, we have $v\in H$.
But $v\in H\cap S$ implies $v\in X^*$ by
Theorem~\ref{saturationTheorem}. This proves (i).
The proof of (ii) is symmetric.
\end{proof}
We can express Corollary~\ref{corollaryChristophe} in a different way.
Let $S$ be an acyclic set and let $X\subset S$ be a bifix code.
Then no nonempty word of $\langle X\rangle$ can be a proper
prefix (or suffix) of a word of $X$. Indeed,
assume that $u\in \langle X\rangle$ is a prefix of a word
of $X$.  Then $u$ is
in  $\langle X\rangle\cap S$ and thus in $X^*$ since $X^*$ is
saturated
in $S$.
This implies $u=1$ or $u\in X$.

We illustrate Theorem~\ref{basisTheorem} in the following example.
\begin{example}\label{exampleBasisJulien}
Let $S$ be as in Example~\ref{exampleJulienAcyclic} (recall that $S$
is not a tree set) and let
$X=S\cap A^2$. We have
\begin{displaymath}
X=\{ab,ac,bc,ca,cd,da\}.
\end{displaymath}
The set $X$ is an $S$-maximal bifix code. It is a basis
of a subgroup of infinite index. Indeed, the minimal
automaton of $X^*$ is represented in Figure~\ref{figureGroupJulien} on the left.
The Stallings automaton of the subgroup $H$
generated by $X$ is obtained by merging $3$ with $4$ and $2$ with
$5$. It is represented in Figure~\ref{figureGroupJulien} on the right.
Since it is not a group automaton,
the subgroup has infinite index (see Proposition~\ref{propStallings}).
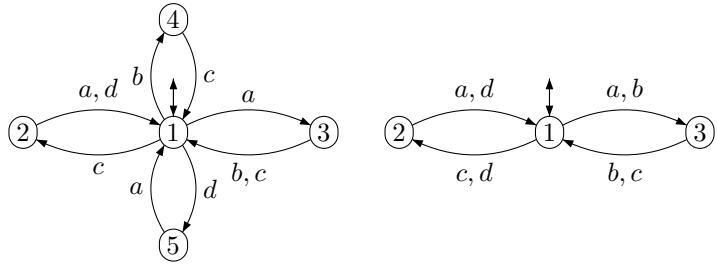
\begin{figure}[hbt]
\gasset{Nadjust=wh}
\centering
\begin{picture}(80,30)
\put(0,10){
\begin{picture}(40,10)(0,-5)
\node(2)(0,0){$2$}
\node[Nmarks=if,iangle=90, fangle=90](1)(20,0){$1$}\node(3)(40,0){$3$}
\node(4)(20,15){$4$}\node(5)(20,-15){$5$}

\drawedge[curvedepth=3](1,2){$c$}\drawedge[curvedepth=3](2,1){$a,d$}
\drawedge[curvedepth=3](1,3){$a$}\drawedge[curvedepth=3](3,1){$b,c$}
\drawedge[curvedepth=3](1,4){$b$}\drawedge[curvedepth=3](4,1){$c$}
\drawedge[curvedepth=3](1,5){$d$}\drawedge[curvedepth=3](5,1){$a$}
\end{picture}
}
\put(50,10){
\begin{picture}(40,10)(0,-5)
\node(2)(0,0){$2$}\node[Nmarks=if,iangle=90, fangle=90](1)(20,0){$1$}\node(3)(40,0){$3$}

\drawedge[curvedepth=3](1,2){$c,d$}\drawedge[curvedepth=3](2,1){$a,d$}
\drawedge[curvedepth=3](1,3){$a,b$}\drawedge[curvedepth=3](3,1){$b,c$}
\end{picture}
}
\end{picture}
\caption{The minimal automaton of $X^*$ and the
Stallings automaton of $\langle X\rangle$.}
\label{figureGroupJulien}
\end{figure}
The set $X$ is a basis of $H$ by Theorem~\ref{basisTheorem}.
This can also be seen by performing Nielsen transformations
on the set $X$ (see~\cite{LyndonSchupp2001} for example). Indeed, replacing $bc$ and $da$ by $bc(ac)^{-1}$ and
 $da(ca)^{-1}$, we obtain
  $X'=\{ab,ac,ba^{-1},ca,cd,dc^{-1}\}$
which is Nielsen reduced. Thus $X'$ is a basis of $H$ and thus also
$X$.

Note that, in agreement with Theorem~\ref{saturationTheorem}, the two
words of length $2$ which are in $H$ but not in $X^*$,
namely $bb$ and $dd$, are not in $S$.

\end{example}

Theorem~\ref{basisTheorem} is false if $X$ is prefix but not bifix, as
shown in the following example.
\begin{example}
Let $S$ be the Fibonacci set and let $X\subset S$ be the
prefix code $X=\{aa,ab,b\}$. Then $a=(ab)b^{-1}$ is in $\langle
X\rangle$ and thus $X$ generates the free group on $A$.
Thus $X$ is not a basis and $X^*\cap S$ is strictly included
in $\langle X\rangle\cap S$ (for example  $a\notin X^*$).
\end{example}

\subsection{Incidence graph}
 Let $X$ be a set, let $P$ be the set of its proper prefixes
and $S$ be the set of its proper suffixes. Set $P'=P\setminus \{1\}$
and $S'=S\setminus \{1\}$.
Recall
from~\cite{BerstelDeFelicePerrinReutenauerRindone2012}
that the incidence graph of $X$ is the undirected graph $G$ defined as
follows.
The set of vertices is the \emph{disjoint union}
of $P'$ and $S'$. The edges of
$G$ are the pairs $(p,s)$ for $p\in P'$ and $s\in S'$
such that $ps\in X$. As in any undirected graph, a connected component
of $G$ is a maximal set of vertices connected by paths.

The following result is proved
in~\cite{BerstelDeFelicePerrinReutenauerRindone2012} in the case of
a Sturmian set (Lemma 6.3.3). We give here a proof in the more
general case of an acyclic set.
We call a path reduced if it does not use equal
consecutive edges.
\begin{proposition}\label{newLemma633}
Let $S$ be an acyclic
set, let $X\subset S$ be a 
bifix code and let $G$ be the incidence graph of $X$. Then the
following
assertions hold.
\begin{enumerate}
\item[\rm(i)]
The graph $G$ is
acyclic.
\item[\rm(ii)]
The intersection of $P'$ (resp. $S'$) with each connected
component
of $G$ is a suffix (resp. prefix) code.
\item[\rm(iii)]
 For every reduced path $(v_1,u_1,\ldots,u_n,v_{n+1})$ in $G$ with
$u_1,\ldots,u_n\in P'$ and $v_1,\ldots,v_{n+1}$ in $S'$,
the longest common prefix of $v_1,v_{n+1}$ is a proper prefix of all
$v_1,\ldots,v_n,v_{n+1}$. 
\item[\rm(iv)]
Symmetrically, for every reduced path $(u_1,v_1,\ldots,v_n,u_{n+1})$ in $G$
with $u_1,\ldots,$ $u_{n+1}\in P'$ and $v_1,\ldots,v_n\in S'$,
the longest common suffix of $u_1,u_{n+1}$ is a proper
suffix of $u_1,u_2,\ldots,u_{n+1}$.
\end{enumerate}
\end{proposition}

\begin{proof}
Assertions (iii) and (iv) imply Assertions (i) and (ii). Indeed, assume that
(iii) holds. Consider
a  reduced path $(v_1,u_1,\ldots,u_n,v_{n+1})$ in $G$ with
$u_1,\ldots,u_n\in P'$ and $v_1,\ldots,v_{n+1}$ in $S'$. If
$v_1=v_{n+1}$,
then the longest common prefix of  $v_1,v_{n+1}$ is not a proper
prefix of them. Thus $G$ is acyclic and (i) holds. Next, if
$v_1$, $v_{n+1}$ are comparable for the prefix order, their longest
common prefix is one of them, a contradiction with (iii) again.
The assertion on $P'$ is proved in an analogous way using  assertion (iv).

We prove  (iii) and (iv)
by induction on $n\ge 1$.

The assertions holds for $n=1$. Indeed, if $u_1v_1,u_1v_2\in X$
and if $v_1\in S'$ is a prefix of
$v_2\in S'$,
then $u_1v_1$ is a prefix of $u_1v_2$, a contradiction with the
hypothesis
that $X$ is a prefix code. The same
holds symmetrically for $u_1v_1,u_2v_1\in X$ since $X$ is a suffix code.

Let $n\ge 2$ and assume that the assertions hold for any path
of length at most $2n-2$. We treat the case of a path
$(v_1,u_1,\ldots,u_n,v_{n+1})$ in $G$ with
$u_1,\ldots,u_n\in P'$ and $v_1,\ldots,v_{n+1}$ in $S'$. The other
case is symmetric.

Let $p$ be the longest common prefix of $v_1$ and $v_{n+1}$.
We may assume that $p$ is nonempty since otherwise the
statement is obviously true.
  Any two elements of the
set $U=\{u_1,\ldots,u_n\}$ are connected by a path of length
at most $2n-2$ (using elements of $\{v_2,\ldots v_n\}$).
Thus, by induction hypothesis, 
$U$ is a suffix code. Similarly, any two elements of the set 
$V=\{v_1,\ldots,v_n\}$ are connected by a path of length
at most $2n-2$ (using elements of $\{u_1,\ldots u_{n-1}\}$).
Thus $V$ is a prefix code.  We cannot have $v_1=p$ since
otherwise, using the fact that $u_np$ is a prefix of $u_nv_{n+1}$
and thus in $S$, the generalized extension
graph $E_{U,V}(\varepsilon)$ would have
the cycle $(p,u_1,v_2,\ldots,u_n,p)$, a contradiction since $E_{U,V}(\varepsilon)$
is acyclic by Proposition~\ref{PropStrongTreeCondition}. Similarly, we
cannot
have $v_{n+1}=p$.

Set $W = p^{-1}V$ and $V' = (V \setminus pW) \cup p$. Since
$V$ is a prefix code and since
$p$ is a proper prefix of $V$, the set $V'$ is
a prefix code.
Suppose that $p$ is not a proper prefix of all $v_2,\dots,v_n$. 
Then there exist $i,j$ with $1\le i<j\le n+1$ such that
$p$ is a proper prefix of $v_i,v_j$ but not of any $v_{i+1},\ldots,v_{j-1}$.
Then $v_{i+1},\dots,v_{j-1} \in V'$ and there is the cycle 
$(p,u_i,v_{i+1},u_{i+1},\dots,v_{j-1},u_{j-1},p)$ in the graph
$E_{U,V'}(\varepsilon)$. 
This is in contradiction with Proposition~\ref{PropStrongTreeCondition} because, 
$V'$ being a prefix code, $E_{U,V'}(\varepsilon)$ is acyclic. 
Thus $p$ is a proper prefix of all $v_2,\dots,v_n$.
\end{proof}

Let $X$ be a bifix code and let $P$ be the set of proper prefixes of
$X$. Consider the equivalence $\theta_X$ on $P$ which is the
transitive closure of the relation formed by the pairs $p,q\in P$ such
that $ps,qs\in X$ for some $s\in A^+$. Such a pair corresponds, when
$p,q\ne1$, to a path $p\rightarrow s\rightarrow q$ in the
incidence graph of $X$.  Thus a class of $\theta_X$ is either reduced
to the empty word or it is the intersection of $P\setminus1$ with a connected
component of the incidence graph of $X$.

The following property relates the equivalence $\theta_X$ with the
right cosets of $H=\langle X\rangle$.
It is Proposition 6.3.5 in~\cite{BerstelDeFelicePerrinReutenauerRindone2012}.
\begin{proposition}\label{propTheta}
  Let $X$ be a bifix code, let $P$ be the set of proper prefixes of
  $X$ and let $H$ be the subgroup generated by $X$. For any $p,q\in
  P$, $p\equiv q\bmod\theta_X$ implies $Hp=Hq$.
\end{proposition}

Let $\A=(P,1,1)$ be the literal automaton of $X^*$.
  We show that the equivalence
$\theta_X$ is compatible with the transitions of the automaton $\A$ in
the following sense.

The following is proved
in~\cite{BerstelDeFelicePerrinReutenauerRindone2012}
 (Lemma 6.3.6 and Lemma 6.4.2) in the case of
a Sturmian set $S$.

\begin{proposition}\label{lemmaCompatible}
  Let $S$ be an acyclic set.
 Let $X\subset S$ be a bifix code and let $P$
  be the set of proper prefixes of $X$.  Let $p,q\in P$ and $a\in
  A$ be such that  $pa,qa\in P\cup X$. Then in the literal automaton
  of $X^*$,
one has $p\equiv q\bmod\theta_X$ if and only if $p\cdot
  a\equiv q\cdot a\bmod\theta_X$.
\end{proposition}

\begin{proof}
Assume first that $p\equiv q\bmod\theta_X$. We may assume that
$p,q$ are nonempty.
Let $(u_0,v_1,
u_1,\ldots,v_n,u_n)$ be a reduced path in the incidence graph $G$ of $X$ with $p=u_0$,
$u_n=q$.  The corresponding words in $X$ are $u_0v_1,
u_1v_1,u_1v_2,\dots,u_nv_n$. 
We may assume that the words $u_i$ are pairwise distinct,
and that the $v_i$ are pairwise distinct.  Moreover, since $pa,qa\in
P\cup X$
 there exist words $v,w$ such that $pav,qaw\in
X$. Set $v_0=av$ and $v_{n+1}=aw$. 

By Proposition~\ref{newLemma633}, $a$ is a proper prefix of
$v_0,v_1,\ldots,v_{n+1}$. Set $v_i=av'_i$ for $0\le i\le n+1$.

If $pa,qa\in P$, then
$(u_0a,v'_1,u_1a,\ldots,v'_n,u_na)$ is a path from $pa$ to $qa$ in $G$.
This shows that $pa\equiv qa\bmod\theta_X$.

Next, suppose that $pa\in X$ and thus that $v_0=a$.
By Proposition~\ref{newLemma633},  we have $w=\varepsilon$
since otherwise $v_0=a$ is  a proper prefix of $v_{n+1}$.
Thus $qa\in X$ and $p\cdot a=q\cdot a$.

Conversely, if $p\cdot a\equiv q\cdot a\bmod\theta_X$, assume first that 
$pa,qa\in P$. Then $pa\equiv qa\bmod\theta_X$ and thus there is a reduced path
$(u_0,v_1,\ldots,v_n,u_n)$ in $G$ with $u_0=pa$ and $u_n=qa$.
By Proposition~\ref{newLemma633}, $a$ is a proper suffix
of $u_1,\ldots,u_n$. Set $u_i=u'_ia$.  Thus $(p,av_1,u'_1,\ldots,q)$
is a path in $G$, showing that $p\equiv q\bmod\theta_X$.

Finally, if $pa,qa\in X$, then $(p,a,q)$ is a path in $G$ and
thus $p\equiv q\bmod\theta_X$.
\end{proof}

\subsection{Coset automaton}

Let $S$ be an acyclic set  and
let $X\subset S$ be a bifix code.
We introduce a new automaton denoted $\B_X$ 
 and called the \emph{coset automaton}\index{coset
  automaton}\index{automaton!coset} of $X$. 
Let $R$ be the set of classes of $\theta_X$ with the class of $1$
still denoted $1$. 
The coset automaton of $X$  is the automaton
$\B_X=(R,1,1)$ with set of states $R$ and transitions induced by the
transitions of the literal automaton $\A=(P,1,1)$ of $X^*$. Formally,
for $r,s\in R$ and $a\in A$, one has $r\cdot a=s$ in the automaton
$\B_X$ if there exist $p$ in the class $r$ and $q$ in the class $s$ such
that $p\cdot a=q$ in the automaton $\A$.

Observe first that the definition is consistent since, by
Proposition~\ref{lemmaCompatible}, if $p\cdot a$ and $p'\cdot a$ are
nonempty and $p, p'$ are in the same class $r$, then $p\cdot a$ and
$p'\cdot a$ are in the same class. 

Observe next that if there is a path from $p$ to $p'$ in  the
automaton $\A$ labeled $w$, then there is a path from the class $r$
of $p$ to the class $r'$ of $p'$ labeled $w$ in $\B_X$.

\begin{figure}[hbt]\label{figureBX}
  \centering
  \begin{picture}(40,10)(0,-5)
    \node[Nmarks=if, iangle=90,fangle=90](1)(0,0){$1$}
    \node[linecolor=red](2)(20,0){\textcolor{red}{$2$}}
    \node[linecolor=blue](3)(40,0){\textcolor{blue}{$3$}}
    \drawedge[curvedepth=3](1,2){$b$}\drawedge[curvedepth=3](2,1){$b$}
    \drawedge[curvedepth=3](2,3){$a$}\drawedge[curvedepth=3](3,2){$a$}
    \drawloop[loopangle=180](1){$a$}\drawloop[loopangle=0](3){$b$}
  \end{picture}
  \caption{The automaton $\B_X$.}
\end{figure}

\begin{example} Let $S$ be the Fibonacci set and
let 
\begin{displaymath}
X=\{a,baab,babaabab,babaabaabab\}.
\end{displaymath}
 The set $X$ is an $S$-maximal
bifix code of $S$-degree $3$ 
(see~\cite{BerstelDeFelicePerrinReutenauerRindone2012}, Example 6.3.1).
The automaton $\B_X$ has three
  states. It is a group automaton. State $2$  is the class containing
  $b$, and state $3$ is the class containing $ba$. The bifix code 
  generating the submonoid  recognized by this automaton is
  $Z=a\cup b(ab^*a)^*b$. 
\end{example}

The following result shows that the coset automaton of $X$ is the
Stallings automaton of the subgroup generated by $X$.

\begin{proposition}\label{lemmaBidet} Let $S$ be an acyclic set, and let 
  $X\subset S$ be a bifix code. The coset automaton $\B_X$ is
  reversible and describes the subgroup generated by $X$.
Moreover $X\subset Z$, where $Z$ is the bifix code
  generating the submonoid recognized by $\B_X$.
\end{proposition}

\begin{proof} Let $\A=(P,1,1)$ be the literal automaton of $X^*$ and set
  $\B_X=(R,1,1)$. By  Proposition~\ref{lemmaCompatible}, the
automaton $\B_X$ is reversible.

Let $Z$ be the bifix code generating the submonoid recognized by
$\B_X$.
To show the inclusion $X\subset Z$, consider a word $x\in X$. There is
a path from $1$ to $1$ labeled $x$ in $\A$, hence also in $\B_X$.
Since
the path in  $\A$ does
not pass by $1$ except at its ends and since
the class of $1$ modulo $\theta_X$ is reduced to $1$, the path in
$\B_X$ does not pass by $1$ except at its ends. Thus $x$ is in $Z$.

Let us finally show that the coset automaton describes the group
$H=\langle X\rangle$. By Proposition~\ref{propGeneratedGroup}, the subgroup 
described
by $\B_X$ is equal to $\langle Z\rangle$. Set $K=\langle Z\rangle$.
Since $X\subset Z$, we have $H\subset K$. To show the converse
inclusion, let us show by induction on the length of $w\in A^*$
that if, for $p,q\in P$, there is a path from the class of $p$ to the
class of $q$ in $\B_X$ with label $w$ then $Hpw=Hq$. By
Proposition~\ref{propTheta}, this holds for $w=1$. Next,
assume
that it is true for $w$ and consider $wa$ with $a\in A$. Assume
that there are states $p,q,r\in P$ such that 
there is a path from the class of $p$ to the class of $q$ in $\B_X$ with
label $w$,
and an edge from the class of $q$ to the class of $r$ in $\B_X$ with
the label $a$. By induction
hypothesis, we have $Hpw=Hq$. Next, by definition of $\B_X$, there
is an $s\equiv q\bmod \theta_X$ such that $s\cdot a\equiv
r\bmod\theta_X$.
If $sa\in P$, then $s\cdot a=sa$, and
by Proposition~\ref{propTheta}, we have $Hs=Hq$ and $Hsa=Hr$.
 Otherwise, $sa\in X\subset H$ and $s\cdot a=r=1$
because the class of $1$ is a singleton and thus $Hqa=Hsa=H=Hr$.
In both cases,  $Hpwa=Hqa=Hsa=Hr$.
This property shows that if $z\in Z$, then
$Hz=H$, that is $z\in H$. Thus $Z\subset H$ and finally $H=K$.
\end{proof}

\subsection{Proof of the main results}

We can now prove Theorem~\ref{basisTheorem}. The proof uses
Proposition~\ref{newLemma633}. We will also use the elementary
fact that if $X$ is a bifix code, and $x,y\in X$ with $x\ne y$, then
 $x$ cannot cancel completely with $y^{-1}$, which means
that $\rho(xy^{-1})$ cannot be a prefix of $x$ or a suffix of $y^{-1}$.
Indeed, if
$xy^{-1}$ is equivalent to a prefix of $x$, then $y$ is a suffix of
$x$
and if $xy^{-1}$ is equivalent to a suffix of $y^{-1}$ then $x$ is a
suffix
of $y$. A symmetric argument holds for $x^{-1}$ and $y$.
\\

\begin{proofof}{ of Theorem~\ref{basisTheorem}}
To prove the necessity of the condition, assume that for some
$w\in S$
the graph $E(w)$ contains a cycle $(a_1,b_1,\ldots,a_p,b_p,a_1)$
with $p\ge 2$, $a_i\in L(w)$ and $b_i\in R(w)$ for
$1\le i\le p$. 
Consider the bifix code $X=AwA\cap S$. Then
$a_1wb_1,a_2wb_1,\ldots,$ $a_pwb_p,a_1wb_p\in X$.
 But
\begin{displaymath}
a_1wb_1(a_2wb_1)^{-1}a_2wb_2\cdots a_pwb_p(a_1wb_p)^{-1}\equiv 1,
\end{displaymath}
contradicting the fact that $X$ is free.

Let us now show the converse. Assume that $S$ is acyclic
and let $X\subset S$ be a bifix code.
Set $Y=X\cup X^{-1}$. Let $y_1,\ldots,y_n\in Y$. We intend to
show that provided $y_iy_{i+1}\not\equiv 1$ for $1\le i<n$, 
we have $y_1\cdots y_n\not\equiv 1$.
We may assume $n\ge 3$.

We say that a sequence $(u_i,v_i,w_i)_{1\le i\le n}$ of elements of
the free group on $A$ is \emph{admissible} with respect to
$y_1,\ldots,y_n$
if the following conditions are satisfied (see Figure~\ref{figurey_i}):
\begin{enumerate}
\item[(i)] $y_i=u_iv_iw_i$ for $1\le
i\le n$,
\item[(ii)] $u_1=w_n=1$ and $v_1,v_n\ne 1$,
\item[(iii)]  $w_iu_{i+1}\equiv 1$ for $1\le i\le  n-1$.
\item[(iv)] For $1\le i<j\le n$, if $v_i,v_j\ne 1$ and $v_k=1$
for $i+1\le k\le j-1$, then $v_iv_j$ is reduced.
\end{enumerate}

Note that if the sequence $(u_i,v_i,w_i)_{1\le i\le n}$ is  admissible 
with respect to $y_1,\ldots,y_n$,
then $y_1\cdots y_n$ is equivalent to the word $v_1\cdots v_n$
 which is a reduced nonempty word. Thus, in particular
$y_1\cdots y_n\not\equiv 1$.

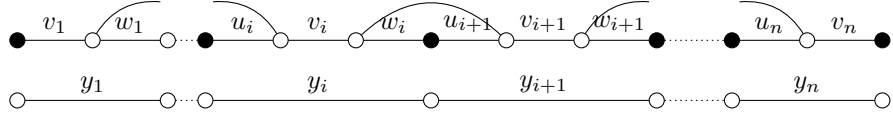
\begin{figure}[hbt]
\centering
\gasset{Nadjust=wh,AHnb=0}
\begin{picture}(100,10)(0,-5)
\node[Nframe=n](wh)(15,5){}
\node[fillcolor=black](l)(-5,0){}\node(ul)(5,0){}\node(wl)(15,0){}

\node[Nframe=n](0h)(20,5){}
\node[fillcolor=black](0)(20,0){}\node(u0)(30,0){}\node(v0)(40,0){}\node[fillcolor=black](w0)(50,0){}
\node(u1)(60,0){}\node(v1)(70,0){}\node[fillcolor=black](w1)(80,0){}\node[Nframe=n](1h)(80,5){}

\node[Nframe=n](rh)(90,5){}
\node[fillcolor=black](r)(90,0){}\node(ur)(100,0){}\node[fillcolor=black](wr)(110,0){}

\drawedge[curvedepth=2](ul,wh){}
\drawedge(l,ul){$v_1$}\drawedge(ul,wl){$w_1$}
\drawedge[dash={0.2 0.5}0](wl,0){}

\drawedge[curvedepth=2](0h,u0){}
\drawedge(0,u0){$u_i$}\drawedge(u0,v0){$v_i$}\drawedge(v0,w0){$w_i$}
\drawedge(w0,u1){$u_{i+1}$}\drawedge(u1,v1){$v_{i+1}$}\drawedge(v1,w1){$w_{i+1}$}

\drawedge[curvedepth=5](v0,u1){}\drawedge[curvedepth=2](v1,1h){}

\drawedge[dash={0.2    0.5}0](w1,r){}\drawedge(r,ur){$u_n$}\drawedge(ur,wr){$v_n$}
\drawedge[curvedepth=2](rh,ur){}

\node(lb)(-5,-8){}\node(lb')(15,-8){}
\node(0b)(20,-8){}\node(1b)(50,-8){}\node(2b)(80,-8){}
\node(3b)(90,-8){}\node(4b)(110,-8){}

\drawedge(lb,lb'){$y_1$}\drawedge(0b,1b){$y_i$}\drawedge(1b,2b){$y_{i+1}$}
\drawedge(3b,4b){$y_n$}

\drawedge[dash={0.2 0.5}0](lb',0b){}
\drawedge[dash={0.2 0.5}0](2b,3b){}
\end{picture}
\caption{The word $y_1\cdots y_{n}$.}\label{figurey_i}
\end{figure}
Let us show by induction on $n$ that for any $y_1,\ldots,y_n$
such that $y_iy_{i+1}\not\equiv 1$ for $1\le i\le n-1$, there exists an admissible
sequence with respect to $y_1\ldots,y_n$.

The property is true for $n=1$. Indeed, we take $u_1=w_1=1$.

Assume that the property is true for $n$. Among the possible
admissible sequences with respect to the $y_1,\ldots,y_n$, we choose
one such that $|v_n|$ is maximal.

Set $v_n=v'_nw'_n$ and
$y_{n+1}=u_{n+1}v_{n+1}$ with $|w'_n|=|u_{n+1}|$ maximal such that
$w'_nu_{n+1}\equiv 1$. Note that $v_{n+1}\ne 1$ since otherwise
$y_{n+1}$ would cancel completely with $y_n$.

If $v'_n\ne 1$, the sequence 
\begin{displaymath}
(1,v_1,w_1),\ldots,(u_{n-1},v_{n-1},w_{n-1}),(u_n,v'_n,w'_n),
(u_{n+1},v_{n+1},1)
\end{displaymath}
 is admissible with respect to $y_1,\ldots,y_{n+1}$.

Otherwise, let $i$ with $1\le i< n$ be the largest integer such that
$v_i\ne 1$. Observe that $w_i,w_{i+1},\ldots, w_{n-1}, w'_n$ are nonempty.
Indeed, if $w_j=1$ with $i\le j\le n-1$, then $u_{j+1}=1$ and thus
$y_{j+1}$ cancels completely with $y_{j+2}$. Next, if $v_n=w'_n=1$,
then $y_n$ cancels completely with $y_{n-1}$.

Assume that $y_i\in X$ (the other case is symmetric).

If $y_{n+1}\in X$ (and thus $n-i$ is odd), then $v_iv_{n+1}$ is
reduced because they are both in $A^*$
and $v_{n+1}\ne 1$ as we have already seen. Thus 
 the sequence 
\begin{displaymath}(1,v_1,w_1),\ldots,(u_{n-1},v_{n-1},w_{n-1}),(u_n,1,w'_n),
(u_{n+1},v_{n+1},1)
\end{displaymath}
 is admissible with respect to $y_1,\ldots,y_{n+1}$.

\begin{figure}[hbt]
\centering
\gasset{Nadjust=wh,AHnb=0}
\begin{picture}(50,45)
\node(1)(0,40){$u_iv_i$}\node(2)(30,40){$w_i=u_{i+1}^{-1}$}
\node(3)(0,30){$u_{i+2}=w_{i+1}^{-1}$}\node(4)(30,30){$w_{i+2}=u_{i+3}^{-1}$}

\node(5)(0,10){$u_n=w_{n-1}^{-1}$}\node(6)(30,10){$v_{n}=u_{n+1}^{-1}$}
\node(7)(0,0){$v_{n+1}^{-1}$}

\drawedge(1,2){}\drawedge(3,2){}
\drawedge(3,4){}\drawedge[dash={0.2 0.5}0](5,4){}
\drawedge(5,6){}\drawedge(7,6){}
\end{picture}
\caption{The graph $G(X)$.}\label{figureGX}
\end{figure}
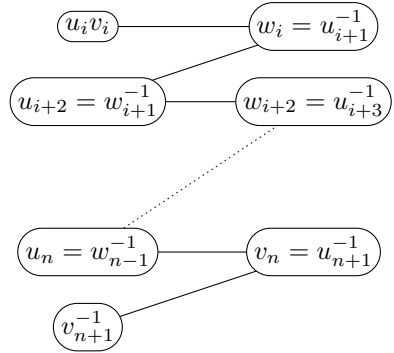
Otherwise, let $s$ be the longest common suffix of $u_iv_i$ and
$v_{n+1}^{-1}$.

There is a path in the incidence graph $G(X)$
from $u_iv_i$ to $v_{n+1}^{-1}$ (see Figure~\ref{figureGX}).
By Proposition~\ref{newLemma633}, $s$ is a proper suffix of
$u_iv_i,w_{i+1}^{-1},\ldots,w_{n-1}^{-1},v_{n+1}^{-1}$. 
This implies that $s^{-1}$ is a proper prefix of
$w_{i+1},\ldots,w_{n-1},v_{n+1}$.

It is not possible that $v_i$ is a suffix of $s$.
Indeed, this would imply that $v_i^{-1}$ is a proper prefix of
$w_{i+1},\ldots,w_{n-1},v_{n+1}$. But then we 
could change the $n-i+1$ last terms of the sequence
$(u_j,v_j,w_j)_{1\le j\le n}$ into $(u_i,1,v_iw_i)$,
$(u_{i+1}v_i^{-1},1,\rho(v_iw_{i+1})),\ldots,$
$(\rho(u_nv_ i^{-1}),v_iv_n,1)$ resulting in an admissible sequence with a
longer $v_n$. 

Thus $s$ is a proper suffix of $v_i$.
Since $s$ is a proper 
suffix of $v_i$ and $v_{n+1}^{-1}$, there are nonempty words $p,q\in A^*$ such
that $v_i=ps$ and $v_{n+1}^{-1}=qs$. Moreover, the word $pq^{-1}$ is
reduced
since $s$ is the longest common suffix of $v_i$ and $v_{n+1}^{-1}$.
Thus we can change the last $n-i+2$ terms of the sequence formed by
$(u_j,v_j,w_j)_{1\le j\le n-1}$ followed by
$(u_n,1,v_n),(u_{n+1},v_{n+1},1)$
into
\begin{displaymath}
(u_i,p,sw_i),
(u_{i+1}s^{-1},1,\rho(sw_{i+1})),\ldots,
(\rho(u_ns^{-1}),1,s v_n),(u_{n+1}s^{-1},q^{-1},1)
\end{displaymath}
 (see Figure~\ref{figurey_is}).
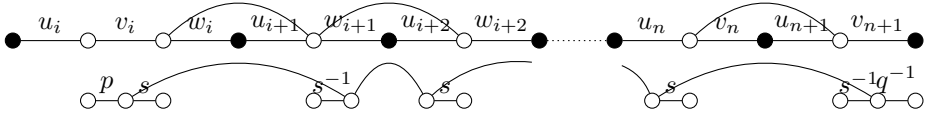
\begin{figure}[hbt]
\centering
\gasset{Nadjust=wh,AHnb=0}
\begin{picture}(120,15)
\node[fillcolor=black](ui)(0,8){}\node(vi)(10,8){}\node(wi)(20,8){}
\node[fillcolor=black](ui+1)(30,8){}\node(wi+1)(40,8){}
\node[fillcolor=black](ui+2)(50,8){}\node(wi+2)(60,8){}
\node[fillcolor=black](ui+3)(70,8){}
\node[fillcolor=black](un)(80,8){}\node(wn)(90,8){}
\node[fillcolor=black](un+1)(100,8){}\node(vn+1)(110,8){}\node[fillcolor=black](r)(120,8){}

\node(p)(10,0){}
\node(s1)(15,0){}\node(s1r)(20,0){}
\node(s2)(40,0){}\node(s2r)(45,0){}
\node(s3)(55,0){}\node(s3r)(60,0){}
\node[Nframe=n](ui+3b)(70,5){}
\node[Nframe=n](unb)(80,5){}
\node(s4)(85,0){}\node(s4r)(90,0){}
\node(s5)(110,0){}\node(s5r)(115,0){}
\node(q)(120,0){}

\drawedge(ui,vi){$u_i$}\drawedge(vi,wi){$v_i$}
\drawedge(wi,ui+1){$w_i$}\drawedge(ui+1,wi+1){$u_{i+1}$}
\drawedge(wi+1,ui+2){$w_{i+1}$}\drawedge(ui+2,wi+2){$u_{i+2}$}
\drawedge(wi+2,ui+3){$w_{i+2}$}
\drawedge[dash={0.2    0.5}0](ui+3,un){}
\drawedge(un,wn){$u_n$}\drawedge(wn,un+1){$v_n$}
\drawedge(un+1,vn+1){$u_{n+1}$}\drawedge(vn+1,r){$v_{n+1}$}
\drawedge[curvedepth=5](wi,wi+1){}\drawedge[curvedepth=5](wi+1,wi+2){}
\drawedge[curvedepth=5](wn,vn+1){}

\drawedge(p,s1){$p$}
\drawedge(s1,s1r){$s$}\drawedge(s2,s2r){$s^{-1}$}
\drawedge(s3,s3r){$s$}
\drawedge(s4,s4r){$s$}
\drawedge(s5,s5r){$s^{-1}$}
\drawedge(s5r,q){$q^{-1}$}
\drawedge[curvedepth=5](s1,s2r){}
\drawedge[curvedepth=5](s2r,s3){}
\drawedge[curvedepth=2](s3,ui+3b){}
\drawedge[curvedepth=1](unb,s4){}
\drawedge[curvedepth=5](s4,s5r){}
\end{picture}
\caption{The word $y_i\cdots y_{n+1}$.}\label{figurey_is}
\end{figure}
Since  the word $pq^{-1}$ is reduced, the new sequence is admissible. 

This shows that $y_1\cdots y_n\not\equiv 1$ for any sequence
$y_1,\ldots,y_n\in X\cup X^{-1}$ such that $y_iy_{i+1}\not\equiv 1$
for $1\le i<n$. Thus $X$ is free.
\end{proofof}
We now give a proof of Theorem~\ref{saturationTheorem}. It uses
 Proposition~\ref{lemmaBidet}.
\\

\begin{proofof}{ of Theorem~\ref{saturationTheorem}}
Let $S$ be an acyclic set and let
$X\subset S$ be a bifix code. We have to prove that $X^*\cap S=\langle
X\rangle\cap S$. Since $X^*\cap S\subset \langle
X\rangle\cap S$, we only need to prove the reverse inclusion.

  Consider the bifix code $Z$
  generating the submonoid re\-cog\-ni\-zed by the coset automaton
  $\B_X$ associated to $X$. Set $Y=Z\cap S$. By
  Theorem~\ref{basisTheorem},
$Y$ is a basis of $\langle Y\rangle$.

  By Proposition~\ref{lemmaBidet}, we have
  $X\subset Z$ and thus $X\subset Y$. 

  Since any reversible automaton is minimal and since the automaton
  $\B_X$ is reversible by Proposition~\ref{lemmaBidet}, it is equal to the
  minimal automaton of $Z^*$.  Let $K$ be the subgroup generated by
  $Z$. By Proposition~\ref{lemmaExercise612}, we have $K\cap A^*=Z^*$.

  This shows that
  \begin{displaymath}
    \langle X\rangle\cap S\subset K\cap S= K\cap A^*\cap S=Z^*\cap S=
    Y^*\cap S\subset Y^*. 
  \end{displaymath}
  The first inclusion holds because $X\subset Z$ implies $\langle
  X\rangle\subset K$.  The last equality follows from the fact that if
  $z_1\cdots z_n\in S$ with $z_1,\ldots ,z_n\in Z$, then each $z_i$ is
  in $S$ (because $S$ is factorial) and hence in $Z\cap S=Y$.  Thus $\langle X\rangle\cap S\subset
  Y^*$. Consider $x\in\langle X\rangle\cap S$. Then
  $x\equiv x_1\cdots x_n$ with $x_i\in X\cup X^{-1}$. But since $\langle X\rangle\cap S\subset
  Y^*$, we have also $x=y_1\cdots y_m$ with $y_i\in Y$. 
Since $X\subset Y$ and since $Y$ is free,
this forces $n=m$ and $x_i=y_i$. Thus all $x_i$ are in $X$
and $x$ is in $X^*$. This shows that
 $\langle X\rangle\cap S\subset X^*$ which was to be proved.
\end{proofof}

The proof of Theorem~\ref{basisTheorem} proves not only that bifix
codes
in acyclic sets are free, but also that, in a sense made
more precise below, the associated reductions are of low complexity.

We first define the \emph{heigth} of $w$ on $A\cup A^{-1}$ equivalent
to $1$ as the least
integer $h$ such that $w$ is a concatenation of words of the form
$w=uvu^{-1}$ where $u$ is a word on $A\cup A^{-1}$ and $v$
is a word of heigth $h-1$ equivalent to $1$. The empty word
is the only word equivalent to $1$ of
heigth $0$. 

We then define the height of an arbitrary word $w$ on $A\cup A^{-1}$
as the least integer $h$ such that $w=z_0v_1z_1\cdots v_nz_n$
with $z_0,\ldots,z_n$ equivalent to $1$ of height at most $h$
and $v_1\cdots v_n$ reduced. 

In this way, any word on $A\cup A^{-1}$ has finite height.
For example, the word $aa^{-1}cbb^{-1}$
has heigth $1$ and $aaa^{-1}bb^{-1}a^{-1}$ has height $2$.
The words of height $0$ are the reduced words.

\begin{proposition}
Let $S$ be an
acyclic set
and let $X\subset S$ be a bifix code. Any word $y=y_1\cdots y_n$
with $y_i\in X\cup X^{-1}$ for $1\le i\le n$
such that  $y_iy_{i+1}\not\equiv 1$
for $1\le i\le n-1$
has height at most $1$.
\end{proposition}
\begin{proof}
 The proof of
Theorem~\ref{basisTheorem} shows that 
$y=z_0v_1z_1\cdots z_{n-1}v_nz_n$ where
\begin{enumerate}
\item[(i)] $z_0,\ldots,z_n$  have height at most $1$,
\item[(ii)] $v_1\cdots v_n$ is reduced.
\end{enumerate}
Thus $y$ has height at most $1$.
\end{proof}
\begin{example}
Let $X$ be as in Example~\ref{exampleBasisJulien}. The word
$bc(ac)^{-1}ab$, which reduces to $bb$,
has height $1$.
\end{example}

\bibliographystyle{plain}
\bibliography{acyclicConnectedTree}

\end{document}